\pdfoutput=1
\RequirePackage{ifpdf}
\ifpdf 
\documentclass[pdftex]{sigma}
\else
\documentclass{sigma}
\fi

\usepackage{amsmath,amsfonts,amssymb,amsthm,amscd,comment,paralist,etoolbox,mathtools,enumitem}
\usepackage{mdwlist}
\usepackage[all]{xy}
\usepackage{tikz}
\usetikzlibrary{arrows}
\usetikzlibrary{decorations.markings}

\makeatletter
\def\namedlabel#1#2{\begingroup
 #2%
 \def\@currentlabel{#2}%
 \phantomsection\label{#1}\endgroup
}

\newcommand\C{\mathbb{C}}
\newcommand\Z{\mathbb{Z}}
\newcommand\Q{\mathbb{Q}}

\newcommand\N{\mathbb{N}}

\newcommand\tr{\mathrm{tr}}

\DeclareMathOperator{\Hom}{Hom}

\DeclareMathOperator{\rad}{rad}

\DeclareMathOperator{\trace}{tr}
\DeclareMathOperator{\Aut}{Aut}
\DeclareMathOperator{\soc}{soc}

\newtheorem{theo}{Theorem}[section]
\newtheorem{prop}[theo]{Proposition}
\newtheorem{lem}[theo]{Lemma}
\newtheorem{cor}[theo]{Corollary}

\theoremstyle{definition}
\newtheorem{defin}[theo]{Definition}
\newtheorem{rem}[theo]{Remark}
\newtheorem{eg}[theo]{Example}

\numberwithin{equation}{section}

\newtoggle{comments}
\newtoggle{details}
\newtoggle{prelimnote}
\newtoggle{detailsnote}

\toggletrue{comments}

\toggletrue{detailsnote}

\iftoggle{comments}{%
 \newcommand{\comments}[1]{
 \begin{center}
 \parbox{6.5 in}{
 \color{red}
 {\footnotesize \textbf{Comments:} #1}
 \color{black}}
 \end{center}}
}{%
 \newcommand{\comments}[1]{}
}

\iftoggle{details}{%
 \newcommand{\details}[1]{
 \ \\
 \color{olivegreen}
 {\footnotesize \textbf{Details:} #1}
 \color{black}
 \\
 }
}{%
 \newcommand{\details}[1]{}
}

\iftoggle{prelimnote}{%
 
}{%
 
}

\begin{document}


\newcommand{\arXivNumber}{1509.08405}

\renewcommand{\PaperNumber}{062}

\FirstPageHeading

\tikzset{->-/.style={decoration={
 markings,
 mark=at position .5 with {\arrow{>}}},postaction={decorate}}}

\ShortArticleName{Skew-Zigzag Algebras}

\ArticleName{Skew-Zigzag Algebras}

\Author{Chad COUTURE}
\AuthorNameForHeading{C.~Couture}
\Address{Department of Mathematics and Statistics, University of Ottawa,\\ 585 King Edward Ave, Ottawa, ON K1N 6N5, Canada}
\Email{\href{mailto:ccout045@uottawa.ca}{ccout045@uottawa.ca}}

\ArticleDates{Received October 02, 2015, in f\/inal form June 17, 2016; Published online June 26, 2016}

\Abstract{We investigate the skew-zigzag algebras introduced by Huerfano and Khovanov. In particular, we relate moduli spaces of such algebras with the cohomology of the corresponding graph.}


\Classification{16G20}

\section{Introduction}

Zigzag algebras were introduced by Huerfano and Khovanov in \cite{MR1872113} in their categorif\/ication of the adjoint representation of simply-laced quantum groups. The zigzag algebra $A(\Gamma)$ is a quotient of the path algebra of the double quiver associated to the Dynkin diagram $\Gamma$ of the quantum group in question. The Grothendieck group of the category of $A(\Gamma)$-modules is then naturally identif\/ied with the weight lattice of $\mathfrak{g}$. Zigzag algebras have also recently appeared in categorif\/ications of the Heisenberg algebra. See \cite{MR2988902} and \cite[Remark~6.2(c)]{RS15}.

In addition to their importance in categorif\/ication, zigzag algebras have a variety of nice features, as pointed out in~\cite{MR1872113}. Examples of such features include the following: they have nondegenerate symmetric trace forms and are quadratic algebras (provided that~$\Gamma$ has at least~3 vertices); if $\Gamma$ is a f\/inite Dynkin diagram, then $A(\Gamma)$ is of f\/inite type and its indecomposable representations are enumerated by roots of $\mathfrak{g}$; and if~$\Gamma$ is bipartite, then the quadratic dual of the $A(\Gamma)$ is the preprojective algebra of $\Gamma$ for a sink-source orientation.

\looseness=-1
\emph{Skew}-zigzag algebras, also introduced in \cite{MR1872113}, are a generalisation of zigzag algebras. They involve coef\/f\/icients $v_{b,c}^a$ for vertices $a$, $b$, $c$ such that~$a$ is connected to both $b$ and $c$. When all coef\/f\/icients are equal to one, they recover the zigzag algebras. The goal of the current paper is to investigate some important properties of skew-zigzag algebras. We provide proofs of some results mentioned in~\cite{MR1872113} without proof, in addition to proving some results that appear to be new.

We begin with a review of the concepts in graph theory necessary for the current paper in Section~\ref{sec:graph-theory}. In Section~\ref{sec:zigzag-algebra}, we recall the def\/inition of the skew-zigzag algebras, f\/ind an explicit basis for them (see Proposition~\ref{basis}), and prove that they are graded symmetric algebras (see Proposition~\ref{symmetric algebra}). In Section~\ref{sec:moduli}, we describe two moduli spaces of skew-zigzag algebras. The f\/irst is the moduli space of skew-zigzag algebras up to isomorphism preserving vertices, which we show to be isomorphic to the graph cohomology of~$\Gamma$ (see Theorem~\ref{thm 1}), as stated in \cite{MR1872113} without proof (see Remark~\ref{rem:HK-moduli-space-statement}). The second is the moduli space of skew-zigzag algebras up to \emph{arbitrary} isomorphism, which we show to be isomorphic to the quotient of the graph cohomology of~$\Gamma$ by a~natural action of the automorphism group of~$\Gamma$ (see Theorem~\ref{theo:moduli-arbitrary-isom}). We should note here that we consider the moduli space only as a group, and do not consider any geometric structure. Finally, in Section~\ref{sec:literature}, we discuss an alternate def\/inition of skew-zigzag algebras that has appeared in the literature. We show that this alternate def\/inition is more restrictive (see Proposition~\ref{not iso}).

\section{Graph theory background} \label{sec:graph-theory}

Recall that a \emph{graph} $\Gamma$ is a pair $(V,E)$ where $V$ is a f\/inite set and $E$ is a set consisting of two element subsets of~$V$. The elements of $V$ are called \emph{vertices} and the elements of~$E$ are called \emph{edges}. Note that this implies that we consider graphs with no loops or multiple edges. We will often depict graphs as diagrams, with a node for each vertex and curves between nodes~$a$ and~$b$ if~$\{a,b\} \in E$. If there exist subsets $A,B \subseteq V$ such that $A \sqcup B=V$ and each element of~$E$ contains one element of $A$ and one element $B$, then $\Gamma$ is said to be \emph{bipartite}.

A \emph{quiver}, $\mathcal{Q}$, is a four-tuple $(\mathcal{Q}_{0},\mathcal{Q}_{1},s,t)$ where $\mathcal{Q}_{0}$ and $\mathcal{Q}_{1}$ are both f\/inite sets and~$s$,~$t$ are set maps from $\mathcal{Q}_{1}$ to $\mathcal{Q}_{0}$. The elements of~$\mathcal{Q}_{0}$ are again called vertices and the elements of~$\mathcal{Q}_{1}$ are called \emph{directed edges}. For each directed edge~$f$, we call~$s(f)$ and $t(f)$ the \emph{source} and \emph{target} of~$f$ (respectively). We will often denote a directed edge with source $a$ and target~$b$ by~$(a \,|\, b)$. Throughout this paper we will consider quivers with no parallel edges, i.e., no directed edges~$f_1$,~$f_2$ such that $s(f_1)=s(f_2)$ and $t(f_1)=t(f_2)$.

\begin{eg}[a graph and a quiver] \label{eg:1}
 The leftmost diagram below represents the graph $(V,E)$ with $V=\{a,b,c,d\}$ and $E=\{\{a,b\},\{b,c\},\{a,c\},\{a,d\}\}$. The rightmost diagram depicts the quiver $\mathcal Q=(\mathcal{Q}_{0},\mathcal{Q}_{1},s,t)$ with $\mathcal{Q}_{0}=\{a,b,c,d\}$, $\mathcal{Q}_{1}=\{(a \,|\, b), (a \,|\, d), (b \,|\, c), (c \,|\, d), (d \,|\, b)\}$ and $s,t \colon \mathcal{Q}_{1} \to \mathcal{Q}_{0}$ given by
 \begin{gather*}
  s((a \,|\, b))=s((a \,|\, d))=a,\qquad s((b \,|\, c))=b,\qquad s((c \,|\, d))=c,\qquad s((d \,|\, b))=d,  \\
  t((a \,|\, b))=t((d \,|\, b))=b, \qquad t((a \,|\, d))=t((c \,|\, d))=d,\qquad t((b \,|\, c))=c.
 \end{gather*}
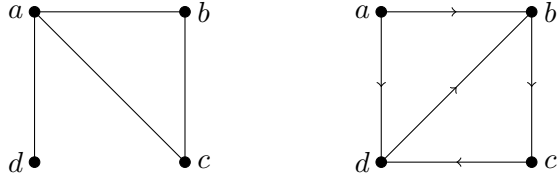
 \begin{center}
 \begin{tikzpicture}
 \coordinate (x) at (0,0);
 \coordinate (y) at (2,0);
 \coordinate (z) at (2,2);
 \coordinate (w) at (0,2);

 \filldraw[black] (0,0) circle (2pt);
 \filldraw[black] (2,0) circle (2pt);
 \filldraw[black] (2,2) circle (2pt);
 \filldraw[black] (0,2) circle (2pt);

 \node at (-0.25,0) {$d$};
 \node at (-0.25,2) {$a$};
 \node at (2.25,2) {$b$};
 \node at (2.25,0) {$c$};

 \draw [-] (w) -- (x);
 \draw [-] (w) -- (y);
 \draw [-] (w) -- (z);
 \draw [-] (z) -- (y);
 \end{tikzpicture}
 \qquad \qquad
 \begin{tikzpicture}
 \coordinate (x) at (0,0);
 \coordinate (y) at (2,0);
 \coordinate (z) at (2,2);
 \coordinate (w) at (0,2);

 \filldraw[black] (0,0) circle (2pt);
 \filldraw[black] (2,0) circle (2pt);
 \filldraw[black] (2,2) circle (2pt);
 \filldraw[black] (0,2) circle (2pt);

 \node at (-0.25,0) {$d$};
 \node at (-0.25,2) {$a$};
 \node at (2.25,2) {$b$};
 \node at (2.25,0) {$c$};

 \draw [->-] (w) -- (z);
 \draw [->-] (w) -- (x);
 \draw [->-] (z) -- (y);
 \draw [->-] (y) -- (x);
 \draw [->-] (x) -- (z);
 \end{tikzpicture}
 \end{center}
\end{eg}

We def\/ine a \emph{path} in a graph $\Gamma$ to be a sequence of vertices $(a_1,a_2,\dots,a_n)$ such that $\{a_i,a_{i+1}\} \in E$ for $i=1, \dots ,n-1$. Analogously, we def\/ine a~path,~$P$, in $\mathcal Q$ to be a sequence of directed edges $(f_1,\dots, f_n)$ such that the source of $f_{i+1}$ is equal to the target of $f_i$ for $1\leq i\leq n-1$. The source and the target of $P$ are the source of $f_1$ and the target of~$f_n$ (respectively). Furthermore, for any path $P=(a_1,a_2,\dots ,a_n)$ (respectively $P=(f_1,\dots , f_m)$), the \emph{length}, $\ell(P)$, of $P$ is equal to $n-1$ (respectively~$m$). We also consider paths of length~0 which start and end at the same vertex~$a$, called \emph{empty paths} or \emph{trivial paths} and denoted $(a)$. We shall use $\mathcal Q_i$ to denote the paths of length $i$ in $\mathcal Q$. In addition, we shall use $(a_1 \,|\, a_2 \,|\, \dots \,|\, a_n)$ $(n \geq 1)$ to denote a~path $(f_1,\dots , f_{n-1})$ such that $s(f_i)=a_i$ and $t(f_i)=a_{i+1}$ for $1\leq i \leq n-1$ in $\mathcal Q$. We say that $\Gamma$ is \emph{connected} if for any vertices $a$ and $b$, there exits a path $P$ between~$a$ and~$b$.

\begin{eg}[paths] \label{eg:paths}
 In the graph of Example~\ref{eg:1}, $(a,b,c,a,d)$ is a path from $a$ to $d$ of length~4. In the quiver of Example~\ref{eg:1}, $(a \,|\, b \,|\, c \,|\, d)$ is a path of length~3. However, $(a \,|\, b \,|\, d)$ is not a path in the quiver of Example~\ref{eg:1}.
\end{eg}

We def\/ine the \emph{double graph} of $\Gamma$, denoted $D\Gamma$, to be the quiver consisting of the vertices of $\Gamma$ and for each edge $\{a,b\}$ of $\Gamma$, $D\Gamma$ has two edges $f_1$ and $f_2$ with $s(f_1)=t(f_2)=a$ and $s(f_2)=t(f_1)=b$.

\begin{eg} [double graph] \label{eg:3}
 Let $\Gamma$ be the graph in Example~\ref{eg:1}. Its double graph, $D\Gamma$, is the following quiver.
 
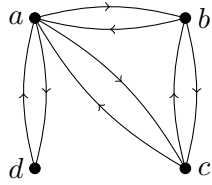
\begin{center}
 \begin{tikzpicture}
 \coordinate (x) at (0,0);
 \coordinate (y) at (2,0);
 \coordinate (z) at (2,2);
 \coordinate (w) at (0,2);

 \draw [->-] (w) to [out=15,in=165] (z);
 \draw [->-] (z) to [out=195,in=-15] (w);
 \draw [->-] (w) to [out=330,in=120] (y);
 \draw [->-] (y) to [out=150,in=300] (w);
 \draw [->-] (w) to [out=285,in=75] (x);
 \draw [->-] (x) to [out=105,in=255] (w);
 \draw [->-] (z) to [out=285,in=75] (y);
 \draw [->-] (y) to [out=105,in=255] (z);

 \filldraw[black] (0,0) circle (2pt);
 \filldraw[black] (2,0) circle (2pt);
 \filldraw[black] (2,2) circle (2pt);
 \filldraw[black] (0,2) circle (2pt);

 \node at (-0.25,0) {$d$};
 \node at (-0.25,2) {$a$};
 \node at (2.25,2) {$b$};
 \node at (2.25,0) {$c$};
 \end{tikzpicture}
 \end{center}
\end{eg}

A path $C=(a_1,a_2,\dots ,a_n)$ in $\Gamma$ is said to be a \emph{cycle} if $a_1 = a_n$. Then we def\/ine $V_C$ and~$E_C$ to be the sets $\{ a_i \,|\, 1 \leq i \leq n-1\}$ and $\{ \{a_i,a_{i+1}\} \,|\, 1 \leq i \leq n-1\}$ respectively. Similarly, a~path $C=(a_1 \,|\, a_2 \,|\, \dots \,|\, a_n)$ in $\mathcal Q$ is said to be a cycle if $a_1 = a_n$. Then we def\/ine $V_C$ and $E_C$ to be the sets $\{ a_i \,|\, 1 \leq i \leq n-1\}$ and $\{ \{a_i,a_{i+1}\} \,|\, 1 \leq i \leq n-1\}$ respectively. For each vertex $a \in V$, we def\/ine the degree of $a$, denoted $\deg (a)$, to be the cardinality of the set $\{e \in E \,|\, a \in e \}$. Furthermore, for any cycle $C$ (in a quiver or a graph), we def\/ine the degree in $C$ of a vertex $a \in V_C$, denoted $\deg _{C} (a)$, to be the cardinality of the set $\{e \in E_{C} \,|\, a \in e \}$. Finally, we say that~$C$ is a \emph{simple cycle} if for all $a \in V_{C}$ we have $\deg _{C} (a)=2$.

We say that a graph $\Gamma'=(V',E')$ is a \emph{subgraph} of $\Gamma=(V,E)$ if $V'\subseteq V$ and~$E'\subseteq E$. If $T$ is a connected graph that does not contain any simple cycles, then $T$ is said to be a \emph{tree}. Let $T=(V_T,E_T)$ be a subgraph of $\Gamma$. If $V_T=V$ and if $T$ is a tree, then $T$ is called a \emph{spanning tree} of~$\Gamma$. Every connected graph has at least one spanning tree (see, for instance, \cite[Section~1.5]{MR2744811}).
If $\Gamma=(V,E)$ is a connected graph and $T=(V_T,E_T)$ is a spanning tree of $\Gamma$, then $\left\vert{E_T}\right\vert=|V|-1$.

\section{Skew-zigzag algebras} \label{sec:zigzag-algebra}

In this section we introduce our main objects of study, the skew-zigzag algebras, and prove some basic facts about them that are stated without proof in the literature. Throughout this section, we f\/ix a f\/ield~$\Bbbk$ and a connected graph $\Gamma = (V,E)$.

\subsection{Def\/initions}

We f\/irst recall the def\/inition of the path algebra of a quiver. We refer the reader to Chapter~2 of~\cite{MR2197389} for further details. The \emph{path algebra} of a quiver $\mathcal Q$, denoted $\Bbbk \mathcal Q$, is the vector space with basis consisting of all paths. We def\/ine the \emph{concatenation} of two paths $(a_1 \,|\, a_2 \,|\, \dots \,|\, a_n)$ and $(a_1' \,|\, a_2' \,|\, \dots \,|\, a_m')$ to be the path $(a_1 \,|\, a_2 \,|\, \dots \,|\, a_n=a_1' \,|\, a_2' \,|\, \dots \,|\, a_m')$ when $a_{n}=a_1'$ and 0 otherwise. Then, we def\/ine the multiplication of two paths $P_1 * P_2$ to be the concatenation of paths. The path algebra is an $\mathbb{N}$-graded algebra, i.e.,
 \begin{gather*}
 \Bbbk\mathcal Q=\bigoplus_{i=0}^{\infty} \Bbbk \mathcal Q_i,
 \end{gather*}
where each $\Bbbk \mathcal Q_i$ is the $\Bbbk$-vector space spanned by all paths of length $i$, and $\Bbbk \mathcal Q_i \Bbbk \mathcal Q_j\subseteq \Bbbk \mathcal Q_{i+j}$ for all~$i$,~$j$. We call an element of $\Bbbk \mathcal Q$ \emph{homogeneous of degree $i$} if it lies in $\Bbbk \mathcal Q_i$. We will be mostly interested in the case where $\mathcal Q = D\Gamma$.

Recall that an ideal $I$ of a graded algebra $A=\bigoplus_{i=0}^{\infty} A_i$ is called \emph{graded} or \emph{homogeneous} if it is generated by homogeneous elements of $A$. Equivalently, $I$ is homogeneous if $I=\bigoplus_{i=0}^{\infty} (I\cap A_i)$.

\begin{defin}[skew-zigzag coef\/f\/icients] \label{coeff defin}
 A set $v = \big(v^a_{b,c} \in \Bbbk \,|\, \{a,b\}, \{a,c\} \in E \big)$ is a collection of \emph{skew-zigzag coefficients} for the graph $\Gamma = (V,E)$ if it satisf\/ies the following three conditions:
 \begin{itemize}\itemsep=0pt
 \item $v^a_{b,b}$=1 for all $\{a,b\} \in E$,
 \item $v^a_{b,c}v^a_{c,b}=1$ for all $\{a,b\},\{a,c\} \in E$, and
 \item $v^a_{b,c}v^a_{c,d}v^a_{d,b}=1$ for all $\{a,b\},\{a,c\},\{a,d\} \in E$.
 \end{itemize}
\end{defin}

From now on, whenever we say that two vertices are \emph{connected}, we shall mean that there is an edge between these two vertices. We are now ready to recall the def\/inition of the algebras def\/ined originally in \cite[p.~527]{MR1872113}.

\begin{defin} [skew-zigzag algebra] \label{skew}
 Let $v$ be a collection of skew-zigzag coef\/f\/icients.
 \begin{itemize}\itemsep=0pt
 \item If $\Gamma$ only contains one vertex, then we def\/ine $A_v(\Gamma$) to be the algebra generated by~1 and~$X$ with $X^2$=0.
 \item If $\Gamma$ contains two vertices, then we def\/ine $A_v(\Gamma)$ to be the quotient algebra of the path algebra of $D\Gamma$ by the two-sided ideal generated by all paths of length greater than two.
 \item If $\Gamma$ has at least three vertices, we def\/ine~$I_v$ to be the ideal generated by
 \begin{enumerate}\itemsep=0pt
 \item[a)] paths of the form $(a_1 \,|\, a_2 \,|\, a_3)$ for all $a_1$, $a_2$, $a_3$ in $\Gamma$ such that~$a_1$ is connected to~$a_2$,~$a_2$ is connected to $a_3$ and $a_1 \neq a_3$,
 \item[b)] elements of the form $(a_1 \,|\, a_2 \,|\, a_1)-v_{a_2,a_3}^{a_1}(a_1 \,|\, a_3 \,|\, a_1)$ for all $a_1 , a_2 , a_3 \in V$ such that~$a_1$ is connected to~$a_2$, $a_3$.
 \end{enumerate}
 We then def\/ine $A_v(\Gamma)$ to be the quotient algebra of the path algebra of $D\Gamma$ by the ideal~$I_v$.
 \end{itemize}
 We call $A_v(\Gamma)$ the \emph{skew-zigzag algebra} of~$\Gamma$. When $v_{a_2,a_3}^{a_1}=1$ for all vertices $a_1$, $a_2$, $a_3$ such that~$a_1$ is connected to both $a_2$ and $a_3$, then we call~$A_v(\Gamma)$ the \emph{zigzag} algebra of $\Gamma$ and denote it~$A(\Gamma)$.
\end{defin}

\begin{rem} \label{grading}
 Notice that $I_v$ is generated by homogeneous elements of $\Bbbk\Gamma$ and hence is a graded ideal. Therefore, the skew-zigzag algebra inherits a grading of its own. More precisely, we have
 \begin{gather*}
 A_v(\Gamma)=\bigoplus_{i=0}^\infty \Bbbk \mathcal Q_i/(I_v\cap \Bbbk \mathcal Q_i)=\bigoplus_{i=0}^\infty (\Bbbk \mathcal Q_i+I_v)/I_v.
 \end{gather*}
 Moreover, if $\Gamma$ contains at least two vertices, $I_v$ is generated by elements of $\Bbbk \mathcal Q_2$. Thus $I_v$ is contained in $\bigoplus_{i \geq 2} \Bbbk \mathcal Q_i$. Consequently, any path of length less than two cannot sit inside $I_v$.
\end{rem}

For a path $P$ in $D\Gamma$, we let $[P]$ denote the equivalence class of $P$ in $A_v(\Gamma)$. Similarly, we shall use the notation $[a_1 \,|\, a_2 \,|\, \dots \,|\, a_n]$ to denote the equivalence class of $(a_1 \,|\, a_2 \,|\, \dots \,|\, a_n)$, $n \geq 1$, in~$A_v(\Gamma)$.

\subsection{Bases}

Our next goal is to describe an explicit basis of the skew-zigzag algebra. We begin with a few technical results. Throughout this subsection, we f\/ix a~collection $v$ of skew-zigzag coef\/f\/icients for the connected graph $\Gamma = (V,E)$.

\begin{lem} \label{coeff}
 Let $a, a_1, \ldots, a_n \in V$ with $n\geq2$, and suppose that~$a$ is connected to $a_1, \ldots, a_n$. Then
 \begin{gather} \label{multiplication}
 v^a_{a_1,a_2}v^a_{a_2,a_3}\cdots v^a_{a_{n-1},a_n}=v^a_{a_1,a_n}.
 \end{gather}
\end{lem}

\begin{proof}
 We shall proceed by induction on $n$. For $n=2$, this trivially holds. Now suppose that~\eqref{multiplication} holds for some integer $n\geq 2$. Let us prove that it holds for $n+1$. We have
 \begin{gather*}
 \prod_{i=1}^nv^a_{a_i,a_{i+1}}=\left(\prod_{i=1}^nv^a_{a_i,a_{i+1}}\right)v^a_{a_{n+1},a_{n-1}}v^a_{a_{n-1},a_{n+1}}
 =\left(\prod_{i=1}^{n-2}v^a_{a_i,a_{i+1}}\right)v^a_{a_{n-1},a_{n+1}}=v^a_{a_1,a_{n+1}},
 \end{gather*}
 where the last equality follows from the induction hypothesis.
\end{proof}

\begin{cor} For all $n\geq 2$, we have
 \begin{gather*}
 v^a_{a_1,a_2}v^a_{a_2,a_3}\cdots v^a_{a_{n-1},a_n}v^a_{a_n,a_1}=1.
 \end{gather*}
\end{cor}

\begin{proof}Suppose $n\geq 2$. By Lemma \ref{coeff}, we have
\begin{gather*}
 v^a_{a_1,a_2}v^a_{a_2,a_3}\cdots v^a_{a_{n-1},a_n}v^a_{a_n,a_1}=v^a_{a_1,a_n}v^a_{a_n,a_1}=1. \tag*{\qed}
\end{gather*}
\renewcommand{\qed}{}
\end{proof}

\begin{lem} \label{lin. ind.}
Let $P_1,\dots,P_n$ be paths in a quiver, no two of which have the same source and target. If $[P_i] \neq 0$ for all $i=0, \dots , n$ then $[P_1],\dots,[P_n]$ are linearly independent.
\end{lem}

\begin{proof}
 Suppose there exist $\alpha_1,\dots, \alpha_n \in \Bbbk$ such that
 \begin{gather*}
 \alpha_1[P_1]+\cdots+\alpha_n[P_n]=0.
 \end{gather*}
 Then, for all $i=1,\dots,n$, we obtain
 \begin{gather*}
 0=\alpha_1[s(P_i)][P_1][t(P_i)]+\cdots+\alpha_i[s(P_i)][P_i][t(P_i)]+\cdots+\alpha_n[s(P_i)][P_n][t(Pi)]=\alpha_i[P_i].
 \end{gather*}
 Hence, we must have $\alpha_i=0$ for all $i=1,\dots,n$. Thus, $[P_1],\dots,[P_n]$ are linearly independent.
\end{proof}

We are now in a position to determine a basis of the skew-zigzag algebra. In particular, this gives us the dimension of the zigzag algebra, which was stated in {\cite[Section~3]{MR1872113}} without proof.

\begin{prop}[basis of the skew-zigzag algebra] \label{basis}
 Recall that $\Gamma=(V,E)$ is a connected graph and $v$ is a collection of skew-zigzag coefficients.
 \begin{itemize}\itemsep=0pt\sloppy
 \item If $\Gamma$ only has one vertex, then $\{1,X\}$ is a basis for $A_v(\Gamma)$.
 \item If $\Gamma$ has two vertices, $a$ and $b$, then $\{[a],[b],[a \,|\, b], [b \,|\, a], [a \,|\, b \,|\, a], [b \,|\, a \,|\, b]\}$ is a basis for~$A_v(\Gamma)$.
 \item If $\Gamma$ has three or more vertices, for all $x\in V$ we define $V_x$ to be set of all vertices that are connected to $x$ and we fix a vertex $y_x \in V_x$. Then
 \begin{gather} \label{J}
 J \coloneqq \{[a], [b \,|\, c], [x \,|\, y_x \,|\, x] \,|\, a,x \in V,\, \{b,c\} \in E\}
 \end{gather}
 is a basis for $A_v(\Gamma)$. In particular, we have $\dim A_v(\Gamma)=2|V|+2|E|$.
 \end{itemize}
\end{prop}

\begin{proof}
 The author would like to thank a referee for bringing into light a much simpler proof of this proposition. 
 
 The f\/irst two claims are obvious. Therefore, we assume that $\Gamma$ has at least three vertices. Recall that the ideal $I_v$ is generated by the set
 \begin{gather*}
 X_v \coloneqq \big\{ (a \,|\, b \,|\, c), (x \,|\, y \,|\, x)-v_{y,z}^{x}(x \,|\, z \,|\, x)  \,|\, \{a,b\}, \{b,c\}, \{x,y\}, \{x,z\} \in E,\, a \ne c \big\},
 \end{gather*}
 and hence
 \begin{gather} \label{eq:in I}
 [a \,|\, b \,|\, c]=0, \qquad [x \,|\, y \,|\, x]=v_{y,z}^{x}[x \,|\, z \,|\, x],
 \end{gather}
 for all $\{a,b\}, \{b,c\}, \{x,y\}, \{x,z\} \in E$ with $a \ne c$. Note that the f\/irst equality implies that any path with three consecutive pairwise distinct vertices is equivalent to zero.
 Now, consider a path of the form $(a \,|\, b \,|\, a \,|\, b)$ where $\{a,b\} \in E$. Since $\Gamma$ is connected and has at least three vertices, either $a$ or $b$ is connected to a third vertex $d \ne a,b$. Suppose $b$ is connected to $d$. (The case that~$a$ is connected to~$d$ is analogous.) Then, \eqref{eq:in I} yields
 \begin{gather*}
 [a \,|\, b \,|\, a \,|\, b] = [a \,|\, b] [b \,|\, a \,|\, b] = v^b_{a,d} [a \,|\, b] [b \,|\, d \,|\, b] = [a \,|\, b \,|\, d] [d \,|\, b] = 0.
 \end{gather*}
 Thus, any path of length 3 or greater has an equivalence class equal to 0. So $A_v(\Gamma)$ only contains elements of degree~0,~1 or~2. It is clear that $\{[a] \,|\, a \in V\}$ and $\{[a \,|\, b] \,|\, \{a,b\} \in E\}$ are bases for $(A_v(\Gamma))_0$ and $(A_v(\Gamma))_1$ respectively since $I_v$ is concentrated in degrees 2 and higher.
 Now, let $x \in V$. Notice that for any $a \in V_x$, we have $[x \,|\, a \,|\, x]=v^x_{a,y_x}[x \,|\, y_x \,|\, x]$. So any element in $([x]A_v(\Gamma)[x])_2$ can be written as some nonzero scalar times $[x \,|\, y_x \,|\, x]$. Thus, $\{[x \,|\, y_x \,|\, x] \,|\, x \in V\}$ is a spanning set for $(A_v(\Gamma))_2$ and so, by Lemma~\ref{lin. ind.}, it is a basis of~$(A_v(\Gamma))_2$. Consequently, $J$~is a basis for~$A_v(\Gamma)$.
 
 Finally, notice that we have
 \begin{gather*}
 J=\{[a]\,|\, a \in V\} \sqcup \{[a \,|\, b] \,|\, \{a,b\} \in E\} \sqcup \{ [x \,|\, y_x \,|\, x] \,|\, x \in V\}.
 \end{gather*}
 Thus, $\left\vert{J}\right\vert = \left\vert{\{[a]\,|\, a \in V\}}\right\vert + \left\vert{\{[a \,|\, b] \,|\, \{a,b\} \in E\}}\right\vert + \left\vert{\{ [x \,|\, y_x \,|\, x] \,|\, x \in V\}}\right\vert=2|V|+2|E|$.
\end{proof}

\subsection{Skew-zigzag algebras as Frobenius algebras}

We begin by recalling the concept of a Frobenius algebra, referring the reader to \cite{Koc04} for further details.
Let $f$ be a bilinear form
\begin{gather*}
 f \colon \ V \times V \to \Bbbk,
\end{gather*}
where $V$ is a vector space of f\/inite dimension over the f\/ield $\Bbbk$. A trace map gives rise to a bilinear form $(x,y) \mapsto \tr(xy)$.

Let $A$ be a $\Bbbk$-algebra. Let $\trace$ be a $\Bbbk$-linear map
 \begin{gather*}
 \trace \colon \  A \to \Bbbk.
 \end{gather*}
We call $\trace$ a \emph{trace} map. A trace map gives rise to a bilinear form $(x,y) \mapsto \tr(xy)$.

\begin{defin} [Frobenius and symmetric algebra]
 Let $A$ be a f\/inite-dimensional unital associative $\Bbbk$-algebra. If there exists a~nondegenerate trace map $\trace \colon A \to \Bbbk$, then $A$ is said to be a~\emph{Frobenius} algebra. Moreover, if there exists a nondegenerate symmetric trace map, then $A$ is said to be a~\emph{symmetric Frobenius} algebra or simply a \emph{symmetric algebra}.
\end{defin}

Recall that $\Gamma=(V,E)$ is a connected graph. Let $v$ be a collection of skew-zigzag coef\/f\/icients, and let $P$ be a path in $D\Gamma$. Throughout this article, we shall def\/ine the source and the target of the equivalence class $[P]$, $s([P])$ and $t([P])$, to be $[s(P)]$ and $[t(P)]$ respectively. If $P_1$ and $P_2$ are both trivial paths or paths of length 1 then we have $[P_1]=[P_2]$ if and only if $P_1=P_2$. If $P_1$ and $P_2$ are paths of length 2 then $[P_1]$ is a scalar multiple of $[P_2]$ if and only if $s(P_1)=s(P_2)$ and $t(P_1)=t(P_2)$. Thus $[s(P)]$ and $[t(P)]$ are well-def\/ined.

If $[P] \neq 0$, we def\/ine the length of $[P]$, denoted $\ell ([P])$, to be $\ell (P)$. If $[P]=0$, then we simply def\/ine $\ell ([P])$ to be 0. Since $I_v$ is generated by homogeneous elements of the same degree, $\ell ([P])$~is well-def\/ined.

Let $J$ be as in \eqref{J} and def\/ine the $\Bbbk$-linear map $\trace \colon A_v(\Gamma) \to \Bbbk$ on the elements of $J$ as follows:
\begin{gather*} \label{tr}
 \trace([P])=
 \begin{cases}
 1 & \text{if } \ell([P])=2,\\
 0 & \text{otherwise.}
 \end{cases}
\end{gather*}
For any path $P=(a \,|\, b \,|\, a)$, we let $v_P=v_{b,y_a}^a$ where $y_a$ is def\/ined as in Proposition \ref{basis}.

\begin{prop} \label{symmetric algebra}
 Recall that $\Gamma=(V,E)$ is a connected graph, and let $v$ be a collection of skew-zigzag coefficients. Then $A_v(\Gamma)$ is a graded Frobenius algebra. In addition, $A(\Gamma)$ is a graded symmetric algebra.
\end{prop}

\begin{proof}
 As noted in \cite[Proposition~1]{MR1872113}, zigzag algebras are symmetric algebras. Furthermore, as stated in {\cite[Section~4.5]{MR1872113}}, skew-zigzag algebras are Frobenius algebras. This follows from the fact that a basic f\/inite-dimensional algebra over an algebraically closed f\/ield is Frobenius if the socle of any projective indecomposable module is simple, and the map $P/\rad(P) \mapsto \soc(P)$ is a bijection onto the set of isomorphism classes of simple modules. The fact that the Frobenius form is as def\/ined above then follows from Propositions 1.10.18 and 3.6.14 of \cite{MR3289041}. It is also possible to prove directly that the trace map def\/ined above has the desired properties.
 \details{
 Consider the trace map, $\trace$, def\/ined in above. For all $a \in V$, def\/ine $y_a$ as in Proposition \ref{basis}. By def\/inition, $\trace$ is $\Bbbk$-linear. Let us prove that it is nondegenerate. Let $x\in A_v(\Gamma)$ be nonzero. Since $\trace$ is a graded map, it suf\/f\/ices to suppose that $x$ is homogeneous. If $x=\sum_{a \in V} \alpha_a[a]$, then pick some $b \in V$ such that $\alpha_b \neq 0$. We then obtain
 \begin{gather*}
 x[b \,|\, y_b \,|\, b] = \left( \sum_{a \in V} \alpha_a[a]\right)[b \,|\, y_b \,|\, b]=\alpha_b[b \,|\, y_b \,|\, b].
 \end{gather*}
 Therefore, $\trace(x[b \,|\, y_b \,|\, b])=\alpha_b \neq 0$. Now, if $x=\sum_{\{a,b\} \in E} \alpha_{a,b}[a \,|\, b]$, then pick some $\{c,d\} \in E$ such that $\alpha_{c,d} \neq 0$. We then obtain
 \begin{gather*}
 x[d \,|\, c]= \left(\sum_{\{a,b\} \in E} \alpha_{a,b}[a \,|\, b]\right)[d \,|\, c]=\alpha_{c,d} [c \,|\, d \,|\, c] = \alpha_{c,d}v^c_{d,y_c}[c \,|\, y_c \,|\, c].
 \end{gather*}
 Hence, $\trace(x[d \,|\, c])= \alpha_{c,d}v^c_{d,y_c} \neq 0$. Finally, suppose that $x= \sum_{a \in V} \alpha_a[a \,|\, y_a \,|\, a]$. Then, pick some $b \in V$ such that $\alpha_b \neq 0$. Multiplying on the right by $[b]$ yields
 \begin{gather*}
 x[b]= \left(\sum_{a \in V} \alpha_a[a \,|\, y_a \,|\, a]\right)[b]=\alpha_b[b \,|\, y_b \,|\, b],
 \end{gather*}
 and so, $\trace(x[b])= \alpha_b \neq 0$. As a result, $\trace$ is nondegenerate and thus, $A_v(\Gamma)$ is a Frobenius algebra.
 
 Now, let us prove that if $P_1$ and $P_2$ are two paths in $D\Gamma$, then we have
 \begin{gather*}
 \trace([P_1][P_2])=
 \begin{cases}
 v_{P_1P_2} & \text{if } \ell([P_1])+\ell([P_2])=2, s([P_1])=t([P_2]) \text{ and }s([P_2])=t([P_1]) \\
 0 & \text{otherwise.}
 \end{cases}
 \end{gather*}
 It is clear that if $P_1$ and $P_2$ are paths such that $\ell([P_1])+\ell([P_2])\neq 2$, $s([P_1])\neq t([P_2])$ or $s([P_2])\neq t([P_1])$, then $\trace([P_1][P_2])=0$. Now assume that $P_1$ and $P_2$ are paths such that $\ell([P_1])+\ell([P_2])=2$, $s([P_1])=t([P_2])$ and $s([P_2])=t([P_1])$. Then $P_1P_2=[a \,|\, b \,|\, a]$ for some $\{a,b\} \in E$. Thus we have
 \begin{gather*}
 \trace([P_1][P_2])=v_{b,y_a}^a \trace([a \,|\, y_a \,|\, a])=v_{b,y_a}^a=v_{P_1P_2}.
 \end{gather*}
If $v=(1)$ then $v_{P_1P_2}=1=v_{P_2P_1}$ for every path $P_1$ and $P_2$ such that $P_1P_2$ is of length 2. Hence, by the above, $\trace$ is symmetric on the elements of $J$. By linearity, $\trace$ is symmetric. Consequently, $A(\Gamma)$ is a symmetric algebra.}
\end{proof}

\section{Moduli spaces of skew-zigzag algebras} \label{sec:moduli}

In this section we describe the moduli spaces of skew-zigzag algebras up to various types of isomorphism. We will see that such moduli spaces are related to the cohomology of the corresponding graph. As noted in the introduction, we will only consider the group structure, and not any geometric structure, on the moduli spaces to be introduced below. Throughout this section we f\/ix a f\/ield $\Bbbk$ that contains square roots.

Let $\Gamma=(V,E)$ be a connected graph and $v$ a collection of skew-zigzag coef\/f\/icients. Let $P = (a_1 \,|\, \ldots \,|\, a_n)$ be a path in $\Gamma$. Def\/ine
\begin{gather*}
 v_{P}^{\rm path} = \prod_{i=2}^{n-1} v_{a_{i-1},a_{i+1}}^{a_i}.
\end{gather*}
We call $v_{P}^{\rm path}$ the \emph{product of $v$ along $P$}. If $P$ is a cycle, then we also def\/ine
\begin{gather*} \label{eq:v-cycle-product}
 v_{P}^{\rm cycle} = v_{a_{n-1},a_2}^{a_1} \prod_{i=2}^{n-1} v_{a_{i-1},a_{i+1}}^{a_i} =v_{a_{n-1},a_2}^{a_1}v_{P}^{\rm path} ,
\end{gather*}
and call $v_{P}^{\rm cycle}$ the \emph{cycle product of $v$ along $P$}. Furthermore, we def\/ine $P^*= (a_n \,|\, \ldots \,|\, a_1)$. It is easily seen that $\big(v_{P}^{\rm path}\big)^{-1} = v_{P^*}^{\rm path}$.

Let $P_1= (a_1 \,|\, \ldots \,|\, a_n)$ and $P_2= (a_n \,|\, \ldots \,|\, a_m)$ be two paths. Notice that we have
\begin{gather}
 v_{P_1P_2}^{\rm path} = \left(\prod_{i=2}^{m-1}v_{a_{i-1},a_{i+1}}^{a_i}\right)= \left(\prod_{i=2}^{n-1}v_{a_{i-1},a_{i+1}}^{a_i}\right)v_{a_{n-1},a_{n+1}}^{a_n} \left(\prod_{i=n+1}^{m-1}v_{a_{i-1},a_{i+1}}^{a_i}\right) \nonumber\\
 \hphantom{v_{P_1P_2}^{\rm path}}{} =v_{P_1}^{\rm path} v_{a_{n-1},a_{n+1}}^{a_n} v_{P_2}^{\rm path}.\label{vproduct}
\end{gather}
If, in addition, $P_1$ and $P_2$ are cycles, then $P_1P_2$ is also a cycle and thus,
\begin{gather}
 v_{P_1P_2}^{\rm cycle} =v_{a_{m-1},a_2}^{a_1}v_{P_1P_2}^{\rm path}
 = v_{a_{m-1},a_2}^{a_1}v_{a_{n-1},a_{n+1}}^{a_n}v_{P_1}^{\rm path}v_{P_2}^{\rm path}\nonumber\\
 \hphantom{v_{P_1P_2}^{\rm cycle}}{}
 =v_{a_{m-1},a_{n+1}}^{a_1}v_{a_{n-1},a_{2}}^{a_1}v_{a_{n-1},a_{n+1}}^{a_n}v_{P_1}^{\rm path}v_{P_2}^{\rm path} \nonumber\\
 \hphantom{v_{P_1P_2}^{\rm cycle}}{}
 =v_{a_{n-1},a_{2}}^{a_1}v_{a_{m-1},a_{n+1}}^{a_n}v_{P_1}^{\rm path}v_{P_2}^{\rm path}
 =v_{P_1}^{\rm cycle}v_{P_2}^{\rm cycle}.\label{P1P2}
\end{gather}
In particular, the second equality of \eqref{P1P2} implies that
\begin{gather} \label{P1P2*c}
 v_{P_1P_2^*}^{\rm cycle}= v_{a_{n+1},a_2}^{a_1}v_{a_{n-1},a_{m-1}}^{a_n}v_{P_1}^{\rm path}v_{P_2^*}^{\rm path}.
\end{gather}

\begin{eg} [product of $v$ along a path and a cycle product]
Consider the graph $\Gamma=(V,E)$ where $V=\{a,b,c,d\}$ and $E=\{\{a,b\},\{a,d\},\{d,c\},\{b,c\},\{b,d\}\}$ and let $\Bbbk =\C$. Its associated double graph is the following quiver.
 
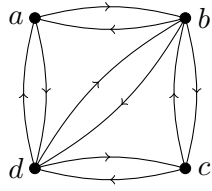
\begin{center}
 \begin{tikzpicture}
 \coordinate (x) at (0,0);
 \coordinate (y) at (2,0);
 \coordinate (z) at (2,2);
 \coordinate (w) at (0,2);

 \draw [->-] (w) to [out=15,in=165] (z);
 \draw [->-] (z) to [out=195,in=-15] (w);
 \draw [->-] (x) to [out=15,in=165] (y);
 \draw [->-] (y) to [out=195,in=-15] (x);
 \draw [->-] (w) to [out=285,in=75] (x);
 \draw [->-] (x) to [out=105,in=255] (w);
 \draw [->-] (z) to [out=285,in=75] (y);
 \draw [->-] (y) to [out=105,in=255] (z);
 \draw [->-] (z) to [out=240,in=30] (x);
 \draw [->-] (x) to [out=60,in=210] (z);

 \filldraw[black] (0,0) circle (2pt);
 \filldraw[black] (2,0) circle (2pt);
 \filldraw[black] (2,2) circle (2pt);
 \filldraw[black] (0,2) circle (2pt);

 \node at (-0.25,0) {$d$};
 \node at (-0.25,2) {$a$};
 \node at (2.25,2) {$b$};
 \node at (2.25,0) {$c$};
 \end{tikzpicture}
 \end{center}
It is easy to show that the following are skew-zigzag coef\/f\/icients for $\Gamma$ using straightforward calculations:
 \begin{gather*}
 v^a_{b,d}=v^c_{b,d}=2, \quad v^a_{d,b}=v^c_{d,b}=1/2, \quad v^d_{a,c}=v^b_{a,c}=5, \quad v^d_{c,a}=v^b_{c,a}=1/5, \quad v^d_{b,c}=v^b_{d,c}=7,\\  v^d_{c,b}=v^b_{c,d}=1/7, \quad v^d_{a,b}=v^b_{a,d}=5/7, \quad v^d_{b,a}=v^b_{d,a}=7/5, \quad v^x_{y,y}=1 \quad \text{for all} \ \ \{x,y\} \in E.
 \end{gather*}
\details{
Clearly we have $v^x_{y,z}*v^x_{z,y}=1$ for all $\{x,y\},\{x,z\} \in E$. In addition, we have
 \begin{gather*}
 v^b_{a,c}*v^b_{c,d}*v^b_{d,a}=5*(1/7)*(7/5)=1, \qquad v^d_{a,c}*v^d_{c,b}*v^d_{b,a}=5*(1/7)*(7/5)=1.
 \end{gather*}
 }
Now, consider the cycles $P_1=(d \,|\, b \,|\, c \,|\, d)$ and $P_2=(d \,|\, a \,|\, b \,|\, d)$. Then, we have
 \begin{gather*}
 v_{P_1}^{\rm path}=14, \quad v_{P_2}^{\rm path}=5/14, \quad v_{P_1P_2}^{\rm path}=1, \quad v_{P_1}^{\rm cycle}=2, \quad v_{P_2}^{\rm cycle}=1/2, \quad \text{and}  \quad v_{P_1P_2}^{\rm cycle}=1.
 \end{gather*}
\end{eg}
\begin{rem}
 We note that $A_v(\Gamma)$ is both a $\Bbbk$-algebra and a $\Bbbk \mathcal{Q}_0$-algebra. Moreover, note that a homomorphism of $\Bbbk \mathcal{Q}_0$-modules is precisely a homomorphism of $\Bbbk$-modules that f\/ixes the vertices.
\end{rem}
\begin{lem} \label{cycle lem}
 Let $\Gamma$ be a connected graph with at least $3$ vertices and $v$ and $u$ be two collections of skew-zigzag coefficients. Suppose that
 \begin{gather*}
 \phi \colon \ A_v (\Gamma) \to A_{u} (\Gamma),
 \end{gather*}
 is an isomorphism of graded algebras such that $\phi([a]) = [a]$ for all $a \in V$. Then
 \begin{gather*} \label{cycle coef}
 u_{P}^{\rm cycle}=v_{P}^{\rm cycle},
 \end{gather*}
 for any cycle $P$.
\end{lem}

\begin{proof}
 For all $\{a,b\} \in E$, we have $\phi ([a \,|\, b])=\alpha_{a,b}[a \,|\, b]$ for some
 $\alpha_{a,b} \in \Bbbk^*$.\details{ Let $\phi([a \,|\, b])=\sum_{\{c,d\} \in E} \alpha_{c,d}[c \,|\, d]$ where $\alpha_{c,d} \in \Bbbk$ for all $\{c,d\} \in E$. Since $\phi$ is an algebra homomorphism, we must have $\phi(xy)= \phi(x) \phi(y)$ for all $x,y \in A_v (\Gamma)$. Hence
 \begin{align*}
 \phi ([a \,|\, b])=\phi ([a][a \,|\, b][b])=\phi ([a])\phi ([a \,|\, b])\phi ([b])=[a]\left( \sum_{\{c,d\} \in E} \alpha_{c,d}[c \,|\, d]\right)[b]=\alpha_{a,b}[a \,|\, b].
 \end{align*}
 Subsequently, $\phi ([a \,|\, b])=\alpha_{a,b}[a \,|\, b]$ for all $\{a,b\} \in E$. Since $\phi$ is an isomorphism, we have $\alpha_{a,b} \neq 0$ for all $\{a,b\} \in E$. }
 In addition, for any $\{a,b\} \in E$ we have
 \begin{gather*}
 \phi([a \,|\, b \,|\, a])=\phi([a \,|\, b][b \,|\, a])=\phi([a \,|\, b])\phi([b \,|\, a])=\alpha_{a,b}[a \,|\, b]\alpha_{b,a}[b \,|\, a]=\alpha_{a,b}\alpha_{b,a}[a \,|\, b \,|\, a].
 \end{gather*}
 Let $a,b,c \in V$ be such that $a$ is connected to both $b$ and $c$. Then
 \begin{align*}
 \alpha_{a,b}\alpha_{b,a} u_{b,c}^a[a \,|\, c \,|\, a]
 & =\alpha_{a,b}\alpha_{b,a} [a \,|\, b \,|\, a]
 =\alpha_{a,b}\alpha_{b,a}[a \,|\, b][b \,|\, a] \\
 & =\phi([a \,|\, b][b \,|\, a])=\phi([a \,|\, b \,|\, a])
 =\phi(v_{b,c}^a[a \,|\, c \,|\, a])
 =v_{b,c}^a\alpha_{a,c}\alpha_{c,a}[a \,|\, c \,|\, a].
 \end{align*}
 Thus, we have
 \begin{gather*} 
 u_{b,c}^a=\frac{\alpha_{a,c}\alpha_{c,a}}{\alpha_{a,b}\alpha_{b,a}} v_{b,c}^a.
 \end{gather*}
 Therefore, if $P=(a_1 \,|\, \ldots \,|\, a_n)$ is a cycle, we have
 \begin{gather*} 
 u_{P}^{\rm cycle}
 = \frac{\alpha_{a_1,a_{n-1}}\alpha_{a_{n-1},a_1}}{\alpha_{a_1,a_2}\alpha_{a_2,a_1}} \left( \prod_{i=2}^{n-2} \frac{\alpha_{a_i,a_{i-1}}\alpha_{a_{i-1},a_i}}{\alpha_{a_i,a_{i+1}}\alpha_{a_{i+1},a_i}}\right) \frac{\alpha_{a_{n-1},a_{n-2}}\alpha_{a_{n-2},a_{n-1}}}{\alpha_{a_{n-1},a_1}\alpha_{a_1,a_{n-1}}}v_{P}^{\rm cycle} =v_{P}^{\rm cycle},
  \end{gather*}
 as required.
\end{proof}

The following proposition shows that the converse to Lemma~\ref{cycle lem} holds.

\begin{prop} \label{iso prop}
 Let $\Gamma=(V,E)$ be a connected graph with at least $3$ vertices and let $v$ and $u$ be two collections of skew-zigzag coefficients. Then $v_{P}^{\rm cycle}=u_{P}^{\rm cycle}$ for every cycle~$P$ if and only if there exists an isomorphism of graded algebras $\phi \colon A_v(\Gamma) \xrightarrow{\cong} A_u(\Gamma)$ such that $\phi([a]) = [a]$ for all $a \in V$.
\end{prop}

\begin{proof}
 $\Rightarrow \colon$
 Fix an edge $\{a,b\} \in E$. For any $\{d,e\} \in E$, consider a path $P=(a_1 \,|\, \ldots \,|\, a_n)$ with $a_1 = a$, $a_{n-1} = d$ and $a_n = e$. Note that since $\Gamma$ is connected, there is at least one such path. Now, choose $\alpha_{d,e} = \alpha_{e,d} \in \Bbbk^*$ such that
 \begin{gather} \label{alphaabs}
 \alpha_{d,e}^2 = u_{b,a_2}^a v_{a_2, b}^a u_{P}^{\rm path} v_{P^*}^{\rm path}.
 \end{gather}
 We will show that since we have $v_{C}^{\rm cycle}=u_{C}^{\rm cycle}$ for any cycle~$C$, \eqref{alphaabs} is independent of the choice of the path. Let $P_1=(a_1 \,|\, \ldots \,|\, a_n)$ and $P_2=(b_1 \,|\, \ldots \,|\, b_m)$ be two paths such that $a=a_1=b_1$, $d=a_{n-1}=b_{m-1}$ and $e=a_n=b_m$. Since $P_1P_2^*$ is a cycle, we have
 \begin{gather*}
 v_{P_1P_2^*}^{\rm cycle}
 = v_{b_2, a_2}^{a_1}v_{a_{n-1}, b_{m-1}}^{a_{n}}v_{P_1}^{\rm path} v_{P_2^*}^{\rm path}
 =v_{b_2,b}^a v_{b,a_2}^a v_{P_1}^{\rm path} v_{P_2^*}^{\rm path},
 \end{gather*}
 where the f\/irst equality follows from \eqref{P1P2*c} and the second equality follows from Lemma \ref{coeff} and the fact that $a_{n-1}=b_{m-1}$. Similarly, we obtain
 \begin{gather*}
 u_{P_1P_2^*}^{\rm cycle} = u_{b_2,b}^a u_{b,a_2}^a u_{P_1}^{\rm path} u_{P_2^*}^{\rm path}.
 \end{gather*}
 Since $v_{P_1P_2^*}^{\rm cycle}=u_{P_1P_2^*}^{\rm cycle}$, we have
 \begin{gather*}
 v_{b_2,b}^a v_{b,a_2}^a v_{P_1}^{\rm path} v_{P_2^*}^{\rm path}
 = u_{b_2,b}^a u_{b,a_2}^a u_{P_1}^{\rm path} u_{P_2^*}^{\rm path}.
 \end{gather*}
 Consequently,
 \begin{gather*}
 u_{b,b_2}^a v_{b_2, b}^a u_{P_2}^{\rm path} v_{P_2^*}^{\rm path}
 = u_{b,a_2}^a v_{a_2, b}^a u_{P_1}^{\rm path} v_{P_1^*}^{\rm path},
 \end{gather*}
 as required. Thus the coef\/f\/icients $\alpha_{d,e}$ are path independent.
Now, def\/ine a map $\phi$ as follows
 \begin{gather*}
 \phi \colon  \ \Bbbk D\Gamma \to \Bbbk D\Gamma, \quad  (d) \mapsto (d) \quad \text{for all} \ \ d \in V, \quad (d \,|\, e) \mapsto \alpha_{d,e}(d \,|\, e) \quad \text{for all} \ \  \{d,e\} \in E.
 \end{gather*}
 By def\/inition this map is $\Bbbk$-linear. We then extend the map to longer paths by requiring it to be an algebra homomorphism. We will now show that $\phi(I_v) \subseteq I_u$.
 Now suppose $\{x,y\}, \{x,z\} \in E$. Let $P_1=(a_1 \,|\, \ldots \,|\, a_n)$ and $P_2=(b_1 \,|\, \ldots \,|\, b_m)$ be two paths such that $a=a_1=b_1$, $x=a_{n-1}=b_{m-1}$, $y=a_n$ and $z=b_m$. We have
 \begin{gather} \label{alpha over alpha}
 \frac{\alpha_{x,y}^2}{\alpha_{x,z}^2} v_{z,y}^x
 = \frac{u_{b,a_2}^a v_{a_2,b}^a u_{P_1}^{\rm path}v_{P_1^*}^{\rm path}} {u_{b,b_2}^a v_{b_2,b}^a u_{P_2}^{\rm path}v_{P_2^*}^{\rm path}}v_{z,y}^x
 = \frac{v_{a_2,b_2}^a v_{P_1^*}^{\rm path}v_{P_2}^{\rm path}} {u_{a_2,b_2}^a u_{P_1^*}^{\rm path}u_{P_2}^{\rm path}}v_{z,y}^x.
 \end{gather}
 The path $P_2(z \,|\, x \,|\, y)P_1^*$ is a cycle. Thus,
 \begin{gather*}
 v_{P_2(z \,|\, x \,|\, y)P_1^*}^{\rm cycle}
 = v_{a_2,b_2}^a v_{P_2(z \,|\, x \,|\, y)P_1^*}^{\rm path}
 = v_{a_2,b_2}^{a} v_{x,x}^z v_{x,x}^y v_{P_2}^{\rm path} v_{(z \,|\, x \,|\, y)}^{\rm path} v_{P_1^*}^{\rm path}
 = v_{a_2,b_2}^{a} v_{P_2}^{\rm path} v_{P_1^*}^{\rm path} v^x_{z,y},
 \end{gather*}
 where the second equality uses \eqref{vproduct}. Similarly, we obtain
 \begin{gather*}
 u_{P_2(z \,|\, x \,|\, y)P_1^*}^{\rm cycle}
 = u_{a_2,b_2}^{a}u_{P_2}^{\rm path}u_{P_1^*}^{\rm path}u_{z,y}^{x}.
 \end{gather*}
 Since we have $v_{P_2(z \,|\, x \,|\, y)P_1^*}^{\rm cycle} = u_{P_2(z \,|\, x \,|\, y)P_1^*}^{\rm cycle}$, we must also have
 \begin{gather} \label{rp}
 v_{a_2,b_2}^{a}v_{P_2}^{\rm path}v_{P_1^*}^{\rm path}v_{z,y}^{x}= u_{a_2,b_2}^{a}u_{P_2}^{\rm path}u_{P_1^*}^{\rm path}u_{z,y}^{x}.
 \end{gather}
 Combining \eqref{alpha over alpha} and \eqref{rp} gives us
 \begin{gather*} \label{relation u v}
 \frac{\alpha_{x,y}^2}{\alpha_{x,z}^2}v_{z,y}^x
 = \frac{v_{a_2,b_2}^a v_{P_1^*}^{\rm path}v_{P_2}^{\rm path}} {u_{a_2,b_2}^a u_{P_1^*}^{\rm path}u_{P_2}^{\rm path}}v_{z,y}^x
 = \frac{u_{a_2,b_2}^a u_{z,y}^{x}u_{P_1^*}^{\rm path}u_{P_2}^{\rm path}} {u_{a_2,b_2}^a u_{P_1^*}^{\rm path}u_{P_2}^{\rm path}} = u_{z,y}^{x}.
 \end{gather*}
 Hence, we can now deduce that $\phi(I_v) \subseteq I_u$, and thus $\phi$ induces an algebra homomorphism
 \begin{gather*}
 \bar{\phi} \colon \ A_v(\Gamma) \to A_u(\Gamma), \quad
 \begin{cases}
 [d] \mapsto [d], & d \in V, \\
 [d \,|\, e] \mapsto \alpha_{d,e} [d \,|\, e], & \{d,e\} \in E,\\
 [d \,|\, e \,|\, d] \mapsto \alpha_{d,e}^2 [d \,|\, e \,|\, d],& \{d,e\} \in E.
 \end{cases}
 \end{gather*}
 \details{To show $\phi(I_v) \subseteq I_u$, we only need to show that the elements that generate~$I_v$ are mapped to elements in~$I_u$. For generators of the form $(d \,|\, e \,|\, f)$ for some $\{d,e\}, \{e,f\} \in E$ with $d \neq f$, we have
 \begin{gather*}
 \phi((d \,|\, e \,|\, f))=\alpha_{d,e}\alpha_{e,f}(d \,|\, e \,|\, f) \in I_u.
 \end{gather*}
 For generators of the form $(d \,|\, e \,|\, d)-v_{e,f}^d(d \,|\, f \,|\, d)$ for some $\{d,e\}, \{d,f\} \in E$ with $e \neq f$, we have
 \begin{gather*}
 \phi((d \,|\, e \,|\, d)-v_{e,f}^d(d \,|\, f \,|\, d))
 = \alpha_{d,e}^2 (d \,|\, e \,|\, d) - v_{e,f}^d\alpha_{d,f}^2 (d \,|\, f \,|\, d)
 = \alpha_{d,e}^2 \left((d \,|\, e \,|\, d) - v_{e,f}^d \frac{\alpha_{d,f}^2}{\alpha_{d,e}^2} (d \,|\, f \,|\, d)\right).
 \end{gather*}
 Since $v_{e,f}^d\frac{\alpha_{d,f}^2}{\alpha_{d,e}^2}= u_{e,f}^d$, we obtain
 \begin{gather*}
 \phi((d \,|\, e \,|\, d)-v_{e,f}^d(d \,|\, f \,|\, d))
 = \alpha_{d,e}^2 \left((d \,|\, e \,|\, d)-u_{e,f}^d(d \,|\, f \,|\, d)\right) \in I_u.
 \end{gather*}
 Consequently, $\phi(I_v) \subseteq I_u$.}
 Since $\phi$ is surjective, $\bar{\phi}$ is also surjective. Since $A_v(\Gamma)$ and $A_u(\Gamma)$ have the same dimension, $\bar{\phi}$~is also injective. Consequently, $\bar{\phi}$ is an isomorphism and hence $A_v(\Gamma) \cong A_u(\Gamma)$.

 The converse of this proposition is Lemma~\ref{cycle lem}.
\end{proof}

We will now introduce the concept of \emph{graph cohomology}.
Let $\Gamma=(V,E)$ be a connected graph and $D\Gamma=(V,E')$ its double graph. Let
\begin{gather*}
 \mathbb{Z}V=\left\{ \sum_{a \in V} \alpha_aa \,|\, \alpha_a \in \mathbb{Z} \text{ for all } a \in V \right\},
\qquad
 \mathbb{Z}E'=\left\{ \sum_{e \in E'} \alpha_ee \,|\, \alpha_e \in \mathbb{Z} \text{ for all } e \in E' \right\}.
\end{gather*}
Def\/ine the map $\delta$ by
\begin{gather*}
 \delta \colon \ \mathbb{Z}E'/\{(a \,|\, b) + (b \,|\, a) \,|\, \{a,b\} \in E\} \to \mathbb{Z}V, \qquad e \mapsto s(e)-t(e),
\end{gather*}
where we extend the map by linearity. For any path $(a_1 \,|\, \ldots \,|\, a_n)$, we associate the element $\sum\limits_{i=1}^{n-1} (a_i|a_{i+1}) \in \mathbb{Z}E'$ to it. Notice that if $a_1=a_n$, then $\delta \left(\sum\limits_{i=1}^{n-1} (a_i|a_{i+1})\right)=0$. Let $\mathcal{C}=\ker \delta$. We call $\mathcal{C}$ the \emph{space of cycles} of $\Gamma$. The space of cycles is a $\mathbb{Z}$-submodule of the free $\mathbb{Z}$-module $\mathbb{Z}E'/\{(a \,|\, b) + (b \,|\, a) \,|\, \{a,b\} \in E\}$ and thus it is a free $\mathbb{Z}$-module. Consequently, it has a $\Z$-basis.

\begin{lem}[\protect{\cite[Section~4.4]{MR3014418}}] \label{dimension}
 Let $\Gamma=(V,E)$ be a connected graph. Then $rank( \mathcal{C}) = |E|-|V|+1$.
\end{lem}

Let $T=(V_T,E_T)$ be a spanning tree of a graph $\Gamma=(V,E)$ and let $e \in V\setminus V_T$. We shall denote by $T+e$ the subgraph of $\Gamma$ with vertex set $V$ and edge set $E_T \cup \{e\}$. Moreover, for any cycle $C=(a_1 \,|\, \ldots \,|\, a_n=a_1)$, we say that an edge $(a \,|\, b)$ is in $C$ if $(a \,|\, b) \in \{(a_i \,|\, a_{i+1})\} \,|\, 1\leq i \leq n-1\}$.

\begin{lem} \label{important}
 Let $\Gamma=(V,E)$ be a connected graph with $ \vert{V} \vert = n$ and $ \vert{E} \vert = m$. There exists a~basis of $\mathcal{C}$, $\mathcal{B}=\{C_1, \ldots ,C_{m-n+1}\}$, such that, for all $1\leq i \leq m-n+1$, there exists $(b_i \,|\, c_i)$ in~$C_i$ such that $(b_i \,|\, c_i)$ is not in~$C_k$ for $k \ne i$.
\end{lem}

\begin{proof} Let $\Gamma=(V,E)$ be a connected graph with $\left\vert{V}\right\vert = n$ and $\left\vert{E}\right\vert = m$. Pick a spanning tree $T=(V_T,E_T)$ of $\Gamma$. Let $E \setminus E_T= \{e_1, \ldots, e_{m-n+1}\}$. Notice that for each $i \in \{1,\ldots, m-n+1\}$, $T+e_i$ contains one simple cycle. Thus, the quiver $D(T+e_i)$ contains two corresponding simple cycles, $C_i$ and $-C_i$. Let
 \begin{gather*}
 \mathcal{B}=\{C_i \,|\, i=1, \ldots, m-n+1 \}.
 \end{gather*}
 Clearly the elements of $\mathcal{B}$ are linearly independent. Moreover, $\left\vert{\mathcal{B}}\right\vert =m-n+1$. Thus, by Lemma~\ref{dimension}, $\mathcal{B}$ is a $\Q$-basis of $\mathcal{C} \otimes_\Z \Q$. But since the coef\/f\/icient of $e_i$ in $C_i$ is one, it follows that~$\mathcal{B}$ is a $\Z$-basis of $\mathcal{C}$. Finally, it is also clear that for all $i=1, \ldots, m-n+1$ we have $(b_i \,|\, c_i)$ in~$C_k$ if and only if $k=i$, where $(b_i \,|\, c_i)$ is the corresponding directed edge of~$e_i$ in~$C_i$.
\end{proof}

Let $\Gamma=(V,E)$ be a connected graph and $D \Gamma =(V,E')$ be its double graph. The \emph{graph cohomology} of $\Gamma$ is def\/ined to be the space of group homomorphisms from $\mathcal{C}$ to $\Bbbk^*$:
\begin{gather*}
 H^1(\Gamma,\Bbbk^*)=\Hom_{\rm group}(\mathcal{C},\Bbbk^*).
\end{gather*}
Note that the operation in this group is pointwise multiplication. 

Let $z \in \mathcal{C}$. Take a representative $\sum\limits_{e \in E'} \alpha_e e$ of $z$ in $\N E'=\Big\{ \sum\limits_{e \in E'} \alpha_ee \,|\, \alpha_e \in \N \text{ for all } e \in E' \Big\}$. Since $z \in \mathcal{C}=\ker \delta$, for each $a \in V$ we have $\sum\limits_{e \in S} \alpha_e =\sum\limits_{e \in T} \alpha_e$ where $S=\{e \in E' \,|\, s(e)=a\}$ and $T=\{e \in E' \,|\, t(e)=a\}$. Thus, we can choose $n_a \in \N$ and vertices $b_{a,1}, \ldots ,b_{a,n_a}, c_{a,1}, \ldots ,c_{a,n_a}$ such that
 \begin{gather*}
 \sum_{a \in V} \sum_{i=1}^{n_a} \left( (b_{a,i} \,|\, a) + (a \,|\, c_{a,i}) \right) = 2z.
 \end{gather*}
Now, for a collection of skew-zigzag coef\/f\/icients $v$, def\/ine
 \begin{gather*}
 f_{v,a} (z) =\prod_{i=1}^{n_a} v^a_{b_{a,i},c_{a,i}}.
 \end{gather*}
 \begin{lem}
 Let $\Gamma=(V,E)$ be a connected graph, $v$ be a collection of skew-zigzag coefficients, $a \in V$ and $z \in \mathcal{C}$. Set $b_{a,1}, \ldots ,b_{a,n_a}$ and $c_{a,1}, \ldots ,c_{a,n_a}$ as in the previous paragraph. Then, $f_{v,a} (z)$ is independent of the order chosen for $b_{a,1}, \ldots ,b_{a,n_a}$ and $c_{a,1}, \ldots ,c_{a,n_a}$ and $f_{v,a} (z)$ is independent of the representative of~$z$ chosen in~$\N E'$.
 \end{lem}
 \begin{proof}
 For the f\/irst claim, it suf\/f\/ices to show that $f_{v,a} (z)$ remains unchanged when we interchange $c_{a,j}$ and $c_{a,j+1}$ for some $j \in \{1,\dots, n_a -1\}$. Indeed, we have
 \begin{gather*}
 v^a_{b_{a,j},c_{a,j+1}}v^a_{b_{a,j+1},c_{a,j}}\prod_{\substack{i=1 \\ i\neq j,j+1}}^{n_a} v^a_{b_{a,i},c_{a,i}} \\
 \qquad{} = v^a_{b_{a,j},c_{a,j}}v^a_{c_{a,j},c_{a,j+1}}v^a_{b_{a,j+1},c_{a,j+1}}v^a_{c_{a,j+1},c_{a,j}}\prod_{\substack{i=1 \\ i\neq j,j+1}}^{n_a} v^a_{b_{a,i},c_{a,i}} =\prod_{i=1}^{n_a} v^a_{b_{a,i},c_{a,i}},
 \end{gather*}
 where the second equality follows from Lemma~\ref{coeff} and the third equality follows from the second condition of Def\/inition~\ref{coeff defin}.

For the second part of the lemma, it suf\/f\/ices to show that $f_{v,a}(z)$ remains unchanged when you remove $b_{a,n}$ and $c_{a,n}$ when
 $b_{a,n}=c_{a,n}$. But this is obvious as $v^a_{b_{a,n},c_{a,n}}=v^a_{b_{a,n},b_{a,n}}=1$.
 \end{proof}

Now, for any collection of skew-zigzag coef\/f\/icients, $v$, def\/ine the map $f_v$ to be
 \begin{gather*}
 f_v \colon \ \mathcal{C} \to \Bbbk^*, \qquad z \mapsto \prod_{a \in V} f_{v,a} (z).
 \end{gather*}
It is clear by the def\/inition of $f_v$ that we have $f_v(z_1 +z_2)=f_v(z_1)f_v(z_2)$ for any $z_1,z_2 \in \mathcal{C}$. Thus, $f_v \in H^1(\Gamma,\Bbbk^*)$. In addition, if $C$ is a cycle in $D\Gamma$, then it is clear that $f_v(C)=v^{\rm cycle}_C$.

Let ${\rm SZC}=\{v \,|\, v \text{ is a collection of skew-zigzag coefficients}\}$. For $u,v \in {\rm SZC}$, def\/ine $u \cdot v$ to be the set of skew-zigzag coef\/f\/icients def\/ined as follows:
 \begin{gather*}
 (u\cdot v)^a_{b,c} = u^a_{b,c}v^a_{b,c},
 \end{gather*}
for all $\{a,b\}, \{a,c\} \in E$. It is straightforward to check that this introduces a group structure on~${\rm SZC}$. We then def\/ine an equivalence relation on ${\rm SZC}$ as follows:
\begin{gather*}
\begin{split}
& v \equiv u \iff \text{there exists an isomorphism } \phi \colon A_v(\Gamma) \to A_u(\Gamma)\\
& \hphantom{v \equiv u \iff{}}{}  \text{such that } \phi([a])=[a] \text{ for all } a \in V.
\end{split}
\end{gather*}
\details{
 Let us prove that this relation is in fact an equivalence relation. Since the identity map f\/ixes the vertices, we clearly have $v \equiv v$, and so $\equiv$ is ref\/lexive. Moreover, if $\phi$ is an isomorphism which f\/ixes the vertices, then $\phi^{-1}$ is an isomorphism which also f\/ixes the vertices. Hence, for any two collections of skew-zigzag coef\/f\/icients~$v$ and $u$, if $v \equiv u$ via some map~$\phi$, then we must have $u \equiv v$ via $\phi^{-1}$. Therefore $\equiv$ is symmetric. Finally, suppose that $\phi$ and $\psi$ are two isomorphism that f\/ix the vertices, then their composition, $\psi \circ \phi$, is also an isomorphism that f\/ixes the vertices. Thus, for any three collections of skew-zigzag coef\/f\/icients $v$, $u$ and $w$, if $v \equiv u$ and $u \equiv w$ via the maps $\phi$ and $\psi$ respectively, then $v \equiv w$ via the isomorphism $\psi \circ \phi$. Hence $\equiv$ is transitive and consequently an equivalence relation.
}
By Lemma~\ref{cycle lem} and Proposition~\ref{iso prop}, we have
\begin{gather*}
 v \equiv u \iff v_P^{\rm cycle} = u_P^\text{cycle} \text{ for every cycle } P \text{ in } \Gamma.
\end{gather*}

Now let
\begin{gather*} \label{eq:Sigma-def}
 \Sigma = \{A_v(\Gamma) \,|\, v \in {\rm SZC}\}.
\end{gather*}
Let $\sim$ be the equivalence relation on $\Sigma$ def\/ined by
\begin{gather*}
 A_v(\Gamma)\sim A_u(\Gamma) \iff \text{there exists an isomorphism } \phi \colon A_v(\Gamma) \to A_u(\Gamma)\\
  \hphantom{A_v(\Gamma)\sim A_u(\Gamma) \iff{}}{} \text{such that } \phi([a])=[a] \ \forall\, a \in V.
\end{gather*}
Let $v \in {\rm SZC}$. From now on, we shall use $[v]$ to denote the equivalence class of $v$ in ${\rm SZC}/_{\equiv}$. Furthermore, we shall use $[A_v(\Gamma)]_{\sim}$ and $[A_v(\Gamma)]_{\cong}$ to denote the equivalence classes of~$A_v(\Gamma)$ in~$\Sigma/_{\sim}$ and~$\Sigma/_{\cong}$ respectively, where $\cong$ denotes isomorphism of graded algebras. Note that $\Sigma/_{\sim}$ and ${\rm SZC}/_{\equiv}$ are naturally isomorphic sets via the map
 \begin{gather} \label{phi}
 \phi \colon  \ \Sigma/_\sim \to {\rm SZC}/_{\equiv}, \qquad [A_v(\Gamma)]_{\sim} \mapsto [v].
 \end{gather}

\begin{theo} \label{thm 1} Let $\Gamma=(V,E)$ be a connected graph. Then we have
 \begin{gather*}
 \Sigma/_\sim \cong {\rm SZC}/_{\equiv} \cong H^1(\Gamma,\Bbbk^*),
 \end{gather*}
 where the first isomorphism is an isomorphism of sets and the second isomorphism is an isomorphism of groups.
\end{theo}

\begin{proof} Def\/ine the map $\psi$ by
 \begin{gather*}
 \psi \colon \ {\rm SZC} \to H^1(\Gamma,\Bbbk^*), \qquad v \mapsto f_v.
 \end{gather*}
 It is clear that this map is a group homomorphism. Let $v,u \in {\rm SZC}$. Then, we have
 \begin{gather*}
 \psi(v)=\psi(u)
 \iff f_v=f_u
 \iff f_v(P)=f_u(P) \text{ for all } P \in \mathcal{C} \\
 \hphantom{ \psi(v)=\psi(u)}{}
 \iff v_P^{\rm cycle}=u_P^{\rm cycle} \text{ for all } P \in \mathcal{C}
 \iff v \equiv u.
 \end{gather*}
 Hence, ${\rm SZC}/\ker \psi={\rm SZC}/_{\equiv}$.

 Let us now prove that $\psi$ is surjective. By Lemma~\ref{important}, there exists a basis of cycles $\mathcal{B}=\{C_1, \ldots, C_{|E|-|V|+1}\}$ of $\mathcal{C}$ such that, for all $1\leq i \leq |E|-|V|+1$, there exists $(b_i \,|\, c_i) \in C_i$ such that $(b_i \,|\, c_i) \in C_k$ if and only if $k=i$.

 Let $f$ be a group homomorphism from $\mathcal{C}$ to $\Bbbk^*$. We def\/ine coef\/f\/icients as follows:
 \begin{gather*}
 v_{b,b}^{a} = 1 \text{ for all } \{a,b\} \in E,\\
 v_{b,c}^{a} = 1 \text{ for all } \{a,b\},\{a,c\} \in E,\qquad a \neq b_i \text{ for all } i=1, \ldots ,|E|-|V|+1.
 \end{gather*}
 Then, for all $i=1, \ldots ,|E|-|V|+1$ we def\/ine
 \begin{gather*}
 1/v_{c_i,a}^{b_i} = v_{a,c_i}^{b_i}  = f(C_i) \text{ for all } a\neq c_k,\ 1\leq k \leq |E|-|V|+1, \\
 v_{c_k,c_i}^{b_i}  = f(C_i)/f(C_k) \text{ for all } k \in \{1,\ldots,|E|-|V|+1\} \text{ such that } b_i=b_k,\\
 v_{a,c}^{b_i}  = 1 \text{ for all } \{b_i,a\},\{b_i,c\} \in E, a \neq c_i \neq c.
 \end{gather*}
 It is clear that the set $v=(v_{b,c}^a \,|\, \{a,b\},\{a,c\} \in E )$ is a collection of skew-zigzag coef\/f\/icients. Moreover, since $(b_i \,|\, c_i)$ is in $C_k$ if and only if $k=i$ for all $i=1, \ldots ,|E|-|V|+1$, it is also clear that $f_v(C_i)=f(C_i)$ for all $i=1, \ldots ,|E|-|V|+1$. Thus, we must have $\psi(v)=f_v=f$. Consequently, $\psi$~is surjective.

Since $\psi$ is surjective, the first isomorphism theorem implies ${\rm SZC}/_{\equiv} \cong H^1(\Gamma,\Bbbk^*)$. The result then follows using~\eqref{phi}.
\end{proof}

\begin{rem} \label{rem:HK-moduli-space-statement}
 In \cite[Section~4]{MR1872113}, the authors state that ``the moduli space of skew-zigzag algebras is naturally isomorphic to $H^1(\Gamma,\mathbb{C}^*)$''. In light of Theorem~\ref{thm 1}, we assume that the moduli space they had in mind was~$\Sigma/\sim$.
\end{rem}

Let $\Gamma=(V,E)$ be a graph. We say that a permutation $\sigma$ of the elements of $V$ is a \emph{graph automorphism} of $\Gamma$ if, for all $a,b \in V$, we have
\begin{gather*}
 \{a,b\} \in E \iff \{\sigma(a),\sigma(b)\} \in E.
\end{gather*}
The identity permutation is called the trivial graph automorphism. A graph is said to be \emph{asymmetric} if it admits only the trivial graph automorphism, otherwise it is said to be \emph{symmetric}.

Suppose $\sigma$ is a graph automorphism of $\Gamma$. For any path $P=(a_1 \,|\, \dots \,|\, a_n)$ we let $\sigma(P)=(\sigma(a_1) \,|\, \dots \,|\, \sigma(a_n))$. Notice that if $P$ is a cycle, then~$\sigma(P)$ is also a cycle. Moreover, if~$P_1$ and~$P_2$ are two paths, then it is clear that $\sigma(P_1P_2)=\sigma(P_1)\sigma(P_2)$. Thus, the map $\sigma$ induces an automorphism of the path algebra $\Bbbk D\Gamma$
\begin{gather*}
 \phi_{\sigma} \colon \ \Bbbk D\Gamma \to \Bbbk D\Gamma, \qquad P \mapsto \sigma(P).
\end{gather*}
Fix a collection of skew-zigzag coef\/f\/icients $v$. We def\/ine $(\sigma v)_{b,c}^a$ for all $\{a,b\},\{a,c\} \in E$ as follows:
\begin{gather*}
 (\sigma v)_{b,c}^a = v_{\sigma^{-1}(b),\sigma^{-1}(c)}^{\sigma^{-1}(a)}.
\end{gather*}
Then we def\/ine $\sigma v = ( (\sigma v)_{b,c}^a \,|\, \{a,b\},\{a,c\} \in E)$. The map $(\sigma, v) \mapsto \sigma v$ def\/ines an action of the group of graph automorphisms on the set of skew-zigzag coef\/f\/icients.
\details{
 For any two graph automorphisms $\sigma_1, \sigma_2$, we have
 \begin{gather*}
 (\sigma_1(\sigma_2 v))_{b,c}^a = (\sigma_2 v)_{\sigma_1^{-1}(b),\sigma_1^{-1}(c)}^{\sigma_1^{-1}(a)} = v_{\sigma_2^{-1}(\sigma_1^{-1}(b)),\sigma_2^{-1}(\sigma_1^{-1}(c))}^{\sigma_2^{-1}(\sigma_1^{-1}(a))} =
 v_{(\sigma_1\sigma_2)^{-1}(b),(\sigma_1\sigma_2)^{-1}(c)}^{(\sigma_1\sigma_2)^{-1}(a)} =
 (\sigma_1\sigma_2 v)_{b,c}^a.
 \end{gather*}
 Thus, $\sigma_1(\sigma_2 v)=(\sigma_1\sigma_2) v$.
}

It is straightforward to verify that $\sigma(I_v) = I_{\sigma v}$.\details{For $a,b,c \in V$ such that $a$ is connected to $b$ and $c$, we have
 \begin{gather*}
 \sigma((a \,|\, b \,|\, a) - v^a_{b,c} (a \,|\, c \,|\, a))
 = (\sigma(a) \,|\, \sigma(b) \,|\, \sigma(a)) - v^a_{b,c} (\sigma(a) \,|\, \sigma(c) \,|\, \sigma(a)) \\
 \hpahntom{\sigma((a \,|\, b \,|\, a) - v^a_{b,c} (a \,|\, c \,|\, a))}{} = (\sigma(a) \,|\, \sigma(b) \,|\, \sigma(a)) - (\sigma v)^{\sigma(a)}_{\sigma(b), \sigma(c)} (\sigma(a) \,|\, \sigma(c) \,|\, \sigma(a)).
 \end{gather*}}
Thus we obtain an isomorphism
\begin{gather} \label{sigma}
 \sigma \colon \ A_v(\Gamma) \to A_{\sigma v}(\Gamma), \qquad [P] \mapsto [\sigma(P)].
\end{gather}

Let
\begin{gather*}
 \psi \colon \ A_v(\Gamma) \to A_u(\Gamma)
\end{gather*}
be an isomorphism of graded algebras where $v$ and $u$ are two collections of skew-zigzag coef\/f\/i\-cients. Suppose that $V=\{a_1, \ldots, a_n\}$. Although it follows from more advance concepts, we will use an elementary approach to prove that vertices must be mapped to vertices. For all $1\leq i \leq n$, let $\psi([a_i]) = \sum\limits_{j=1} ^n \alpha_{ij}[a_j]$, where $\alpha_{ij} \in \Bbbk$ for all $1\leq i,j \leq n$. For any two vertices $a_i, a_k \in V$, $1\leq i,k \leq n$, $i\neq k$, we have
\begin{gather*}
 0=\psi(0)=\psi([a_i][a_k])=\left(\sum_{j=1} ^n \alpha_{ij}[a_j]\right)\left(\sum_{j=1} ^n \alpha_{kj}[a_j]\right) =
 \sum_{j=1} ^n \alpha_{ij}\alpha_{kj}[a_j].
\end{gather*}
Thus, for any $j=1,\ldots ,n$, at least one of $\alpha_{ij}$ or $\alpha_{kj}$ is 0. Hence, if $\alpha_{ij}\neq 0$ for some $1\leq i,j \leq n$, then $\alpha_{kj}=0$ for all $1\leq k \leq n$, $k\neq i$. Moreover, since $\psi$ is injective, $\psi([a_i]) \neq 0$ for all $i=1, \ldots, n$, and so there exists $j_i \in \{1,\ldots ,n\}$ such that $\alpha_{ij_i}\neq 0$ and thus $\alpha_{kj_i}= 0$ for all $1\leq k \leq n$, $k\neq i$. Therefore, for any $1\leq i,k \leq n$, $i\neq k$ we have $j_i \neq j_k$. As a result, we must have $\alpha_{ij}=0$ for all $1\leq i,j \leq n$, $j\neq j_i$. Consequently, we must have $\psi([a_i])=\alpha_{ij_i}[a_{j_i}]$ for some $j_i \in \{1,\ldots ,n\}$. Furthermore, notice that we have
 \begin{gather*}
 \phi([a_i])=\phi([a_i][a_i])=\phi([a_i])\phi([a_i])=\alpha_{ij_i}\alpha_{ij_i}[a_{j_i}]=\alpha_{ij_i}\phi([a_i]),
 \end{gather*}
for all $1\leq i \leq n$. Hence we must have $\alpha_{ij_i}=1$. Therefore, any isomorphism $\psi \colon A_v(\Gamma) \to A_u(\Gamma)$ induces a graph automorphism
 \begin{gather*}
 \sigma_{\psi} \colon \  V \to V, \qquad a_i \mapsto a_{j_i}.
 \end{gather*}
Then, we obtain the isomorphism
 \begin{gather*}
 \sigma_{\psi}^{-1} \circ \psi \colon  \ A_v(\Gamma) \to A_{\sigma^{-1}u}(\Gamma),
 \end{gather*}
Therefore
\begin{gather} \label{composition decomp}
 \psi = \sigma_{\psi} \circ \big(\sigma_{\psi}^{-1} \circ \psi\big),\qquad \text{where } \sigma_{\psi}^{-1} \circ \psi([a]) = [a] \text{ for all } a \in V.
\end{gather}

\begin{lem} \label{equiv}
 Let $\Gamma$ be a connected asymmetric graph with at least three vertices, and let~$v$ and~$u$ be two collections of skew-zigzag coefficients. Then $A_v(\Gamma)\cong A_u(\Gamma)$ as graded algebras if and only if $v \equiv u$.
\end{lem}

\begin{proof}
 Let $\Gamma$ be a connected asymmetric graph with at least three vertices, and $v$ and $u$ be two collections of skew-zigzag coef\/f\/icients. Suppose that we have an isomorphism $\phi \colon A_v(\Gamma)\to A_u(\Gamma)$. Then in particular the map
 \begin{gather*}
 \sigma_{\phi} \colon \ V \to V, \qquad a \mapsto s\left(\phi([a])\right),
 \end{gather*}
 is a graph automorphism. Since $\Gamma$ is asymmetric, we must have $\sigma_{\phi}(a)=a$ for all $a \in V$. Thus $\phi([a])=[a]$ for all $a \in V$. Hence $v \equiv u$ by Lemma~\ref{cycle lem}.

 The reverse implication follows from the def\/inition of the equivalence relation.
\end{proof}

Let $\Aut (\Gamma)$ be the group of graph automorphisms of $\Gamma$. We def\/ine a group action of $\Aut (\Gamma)$ on $H^1(\Gamma, \Bbbk^*)$ by
\begin{gather*}
 (\sigma f)(C)=f\big(\sigma^{-1}(C)\big),\qquad \sigma \in \Aut (\Gamma), \quad f \in H^1(\Gamma, \Bbbk^*),\quad C \in \mathcal{C}.
\end{gather*}
\details{
 For any $\sigma_1,\sigma_2 \in \Aut (\Gamma)$ and any $f\in H^1(\Gamma, \Bbbk^*)$ we have
 \begin{gather*}
 (\sigma_1(\sigma_2 f))(C)=(\sigma_2 f)\big(\sigma_1^{-1}(C)\big)=f\big(\sigma_2^{-1}\sigma_1^{-1}(C)\big)=f\big((\sigma_1\sigma_2)^{-1}(C)\big) =(\sigma_1\sigma_2f)(C)
 \end{gather*}
 for all $C \in \mathcal{C}$.
}

For any $\sigma \in \Aut (\Gamma)$ and any collection of skew-zigzag coef\/f\/icients $v$ we have an isomorphism given by~\eqref{sigma} $\sigma \colon A_v(\Gamma) \to A_{\sigma v}(\Gamma)$. It is clear that $A_v(\Gamma) \mapsto A_{\sigma v}(\Gamma)$ def\/ines a group action of $\Aut (\Gamma)$ on $\Sigma$ and that it preserves the equivalence relation~$\sim$.\details{ To show that $A_v(\Gamma) \mapsto A_{\sigma v}(\Gamma)$ preserves the equivalence relation $\sim$, let $v$ and $u$ be two collections of skew-zigzag coef\/f\/icients. Suppose $A_v(\Gamma) \sim A_u(\Gamma)$. Then there exists an isomorphism $\phi \colon A_v(\Gamma) \to A_u(\Gamma)$ that f\/ixes the vertices. So, we have
 \begin{gather*}
 A_{\sigma v}(\Gamma) \xrightarrow{\sigma^{-1}} A_v(\Gamma) \xrightarrow{\phi} A_u(\Gamma) \xrightarrow{\sigma} A_{\sigma u}(\Gamma),
 \end{gather*}
 where the map $\sigma \circ \phi \circ \sigma^{-1}$ is an isomorphism from $ A_{\sigma v}(\Gamma)$ to $ A_{\sigma u}(\Gamma)$ that f\/ixes the vertices. Hence $A_{\sigma v}(\Gamma) \sim A_{\sigma u}(\Gamma)$.

 Now assume that we have $A_{\sigma v}(\Gamma) \sim A_{\sigma u}(\Gamma)$ for some collections of skew-zigzag coef\/f\/icients and $\sigma \in \Aut(\Gamma)$. So there exists an isomorphism $\phi \colon A_{\sigma v}(\Gamma) \to A_{\sigma u}(\Gamma)$ that f\/ixes the vertices. Thus, we have
 \begin{gather*}
 A_v(\Gamma) \xrightarrow{\sigma} A_{\sigma v}(\Gamma) \xrightarrow{\phi} A_{\sigma u}(\Gamma) \xrightarrow{\sigma^{-1}} A_u(\Gamma),
 \end{gather*}
 where $\sigma^{-1} \circ \phi \circ \sigma$ is an isomorphism from $ A_v(\Gamma)$ to $ A_u(\Gamma)$ that f\/ixes the vertices. Therefore $A_v(\Gamma) \sim A_u(\Gamma)$. Consequently, the action preserves the equivalence relation $\sim$.
}
Therefore, we have an induced action of $\Aut (\Gamma)$ on $\Sigma/_{\sim}$, given by $X \mapsto \sigma X$ for all $X \in \Sigma/_{\sim}$, where $\sigma X = \{\sigma x \,|\, x \in X\}$.

\begin{lem} \label{lem 1}
 For any graph $\Gamma$ we have the following isomorphism of sets
 \begin{gather*}
 (\Sigma/_{\sim})/\Aut(\Gamma) \cong \Sigma/_{\cong}.
 \end{gather*}
\end{lem}

\begin{proof}
 Def\/ine
 \begin{gather*}
 \alpha \colon \  \Sigma/_{\sim} \to \Sigma/_{\cong},\qquad [A_v(\Gamma)]_{\sim} \mapsto [A_v(\Gamma)]_{\cong}.
 \end{gather*}
 Clearly this is a well-def\/ined surjective map. Moreover, notice that if~$v$ and~$u$ are two collections of skew-zigzag coef\/f\/icients, then we have
 \begin{gather*}
 \alpha([A_v(\Gamma)]_{\sim})=\alpha([A_u(\Gamma)]_{\sim}) \iff [A_v(\Gamma)]_{\cong}=[A_u(\Gamma)]_{\cong}\\
 \hphantom{\alpha([A_v(\Gamma)]_{\sim})=\alpha([A_u(\Gamma)]_{\sim})}{} \iff \text{there exists an isomorphism } \Phi \colon A_v(\Gamma) \to A_u(\Gamma).
 \end{gather*}
 Recall from \eqref{composition decomp} that every isomorphism between skew-zigzag algebras can be written as the composition of a graph automorphism and an isomorphism that f\/ixes the vertices. Thus, there exists $\sigma \in \Aut(\Gamma)$ and an isomorphism $\gamma$ that f\/ixes the vertices such that
 \begin{gather*}
 \Phi = \sigma \circ \gamma \colon \ A_v(\Gamma) \xrightarrow{\gamma} A_{\sigma^{-1}u}(\Gamma) \xrightarrow{\sigma} A_u(\Gamma).
 \end{gather*}
 Since $\gamma$ f\/ixes the vertices, we have $[A_v(\Gamma)]_{\sim}=[A_{\sigma^{-1}u}(\Gamma)]_{\sim}$ and so,
 \begin{gather*}
 \sigma[A_v(\Gamma)]_{\sim}= \sigma[A_{\sigma^{-1}u}(\Gamma)]_{\sim}=[A_u(\Gamma)]_{\sim}.
 \end{gather*}
 Now, let $v$ and $u$ be two collections of skew-zigzag coef\/f\/icients such that there exists $\sigma \in \Aut(\Gamma)$ such that $\sigma[A_v(\Gamma)]_{\sim}=[A_u(\Gamma)]_{\sim}$. Then, since $A_{\sigma^{-1}u}(\Gamma) \in [A_v(\Gamma)]_{\sim}$, there exists an isomor\-phism~$\gamma$ that f\/ixes the vertices such that we have the following
 \begin{gather*}
 A_v(\Gamma) \xrightarrow{\gamma} A_{\sigma^{-1}u}(\Gamma) \xrightarrow{\sigma} A_u(\Gamma).
 \end{gather*}
 Since $\sigma \circ \gamma$ is an isomorphism, we must have $[A_v(\Gamma)]_{\cong}=[A_u(\Gamma)]_{\cong}$ and thus, $\alpha([A_v(\Gamma)]_{\sim})=\alpha([A_u(\Gamma)]_{\sim})$.

 Consequently, for any $X_1, X_2 \in \Sigma/_{\sim}$, we have
 \begin{gather*}
 \alpha(X_1)=\alpha(X_2) \iff \text{ there exists } \sigma \in \Aut(\Gamma) \text{ such that } \sigma X_1 = X_2.
 \end{gather*}
 Hence, we f\/inally obtain $(\Sigma/_{\sim})/\Aut(\Gamma) \cong \Sigma/_{\cong}$.
\end{proof}

\begin{theo} \label{theo:moduli-arbitrary-isom}
 For any graph $\Gamma$ we have the following isomorphism of sets
 \begin{gather*}
 \Sigma/_{\cong} \cong H^1(\Gamma, \Bbbk^*)/\Aut(\Gamma).
 \end{gather*}
\end{theo}

\begin{proof}
 By Theorem \ref{thm 1}, we know that the map
 \begin{gather*}
 \psi \circ \phi \colon \ \Sigma/_{\sim} \to H^1(\Gamma, \Bbbk^*), \qquad [A_v(\Gamma)]_{\sim} \mapsto f_v,
 \end{gather*}
 is an isomorphism, where the map $\phi$ is as in \eqref{phi} and $\psi \colon {\rm SZC}/_{\equiv} \to H^1(\Gamma, \Bbbk^*),$ $[v] \mapsto f_v$. Let us now prove that it preserves the $\Aut (\Gamma)$-action, i.e., for all $\sigma \in \Aut (\Gamma)$ and $v \in {\rm SZC}$, we have $(\psi \circ \phi)(\sigma [A_v(\Gamma)]_{\sim})=\sigma (\psi \circ \phi)([A_v(\Gamma)]_{\sim})$. We have
 \begin{gather*}
 (\psi \circ \phi)(\sigma [A_v(\Gamma)]_{\sim})= \psi([\sigma v])=f_{\sigma v},
 \end{gather*}
 and
 \begin{gather*}
 \sigma (\psi \circ \phi)([A_v(\Gamma)]_{\sim})= \sigma (\psi([v]))=\sigma (f_v).
 \end{gather*}
 Let $C=(a_1 \,|\, \ldots \,|\, a_n) \in C$ be a cycle. Then,
 \begin{gather*}
 \sigma (f_v)(C)=f_v\big(\sigma^{-1}(C)\big)= f_v\big( \big(\sigma^{-1}(a_1) \,|\, \ldots \,|\, \sigma^{-1}(a_n)\big) \big)\\
 \hphantom{\sigma (f_v)(C)=f_v\big(\sigma^{-1}(C)\big)}{}
 = v_{\sigma^{-1}(a_{n-1}),\sigma^{-1}(a_2)}^{\sigma^{-1}(a_1)}\cdots v_{\sigma^{-1}(a_{n-2}),\sigma^{-1}(a_n)}^{\sigma^{-1}(a_{n-1})} = f_{\sigma v}(C).
 \end{gather*}
 Hence $\sigma (f_v)=f_{\sigma v}$. Consequently, we must have $(\psi \circ \phi)(\sigma [A_v(\Gamma)]_{\sim})=\sigma (\psi \circ \phi)([A_v(\Gamma)]_{\sim})$. So $\psi \circ \phi$ is an $\Aut (\Gamma)$-set isomorphism and thus,
 \begin{gather*}
 (\Sigma /_{\sim})/\Aut(\Gamma) \cong H^1(\Gamma,\Bbbk^*)/\Aut(\Gamma).
 \end{gather*}
 Lemma \ref{lem 1} then yields $\Sigma/_{\cong} \cong H^1(\Gamma,\Bbbk^*)/\Aut(\Gamma)$.
\end{proof}

In \cite[Section~4]{MR1872113} the authors state the following result without proof.

\begin{cor}
 If $\Gamma$ is a tree then all of its skew-zigzag algebras are isomorphic.
\end{cor}

\begin{proof}
 The follows immediately from Theorem~\ref{theo:moduli-arbitrary-isom} and the fact that trees have trivial graph cohomology.
\end{proof}

\begin{eg}
 Consider the graph $\Gamma=(V,E)$ given by $V=\{a,b,c\}$ and $E=\{\{a,b\},\{b,c\}$, $\{a,c\}\}$. The associated double graph is
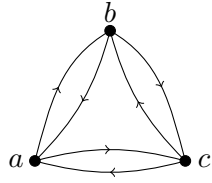
\begin{center}
 \begin{tikzpicture}
 \coordinate (x) at (0,0);
 \coordinate (y) at (2,0);
 \coordinate (z) at (1,1.732);

 \draw [->-] (x) to [out=75,in=210] (z);
 \draw [->-] (z) to [out=255,in=45] (x);
 \draw [->-] (x) to [out=15,in=165] (y);
 \draw [->-] (y) to [out=195,in=-15] (x);
 \draw [->-] (z) to [out=-30,in=105] (y);
 \draw [->-] (y) to [out=135,in=285] (z);

 \filldraw[black] (0,0) circle (2pt);
 \filldraw[black] (2,0) circle (2pt);
 \filldraw[black] (1,1.732) circle (2pt);

 \node at (-0.25,0) {$a$};
 \node at (1,1.982) {$b$};
 \node at (2.25,0) {$c$};
 \end{tikzpicture}
\end{center}
 By Proposition \ref{iso prop}, if $u,v \in {\rm SZC}$, then $A_u(\Gamma)\sim A_v(\Gamma)$ if and only if $u^{\rm cycle}_P = v^{\rm cycle}_P$ for every cycle $P$. Furthermore, recall that by \eqref{composition decomp}, if $u,v \in {\rm SZC}$ are such that $A_v(\Gamma)\cong A_u(\Gamma)$ via an isomorphism $\phi$, then $\phi$ can be written as the composition of a graph automorphism $\sigma$ with an isomorphism $\gamma$ that f\/ixes the vertices:
 \begin{gather*}
 \sigma \circ \gamma \colon \  A_v(\Gamma) \xrightarrow{\gamma} A_{\sigma^{-1}u}(\Gamma) \xrightarrow{\sigma} A_u(\Gamma).
 \end{gather*}
 Thus, if $P$ is a cycle, then $v^{\rm cycle}_P=u^{\rm cycle}_{\sigma^{-1}(P)}$. Conversely, if there exists $\sigma \in \Aut (\Gamma)$, such that $v^{\rm cycle}_P=u^{\rm cycle}_{\sigma^{-1}(P)}$ for every cycle $P$, then we have $A_v(\Gamma)\cong A_{\sigma^{-1}u}(\Gamma)$ by Proposition~\ref{iso prop} and $A_{\sigma^{-1}u}(\Gamma) \cong A_u(\Gamma)$ by \eqref{sigma}. Thus, $A_v(\Gamma)\cong A_u(\Gamma)$. Consequently, $A_v(\Gamma)\cong A_u(\Gamma)$ if and only if there exists some $\sigma \in \Aut(\Gamma)$ such that $v_P^\text{cycle} = u_{\sigma^{-1}(P)}^\text{cycle}$ for all cycles $P$. Therefore, in order to compute $\Sigma/_\sim$ and $\Sigma/_{\cong }$, it suf\/f\/ices to consider the cycle products of $u$ and $v$ ($u,v \in {\rm SZC}$) along every cycle. 
 
 It is clear that the $C=(a \,|\, b \,|\, c \,|\, a)$ yields a basis for the cycle space $\mathcal{C}$. So, if $u,v \in {\rm SZC}$, then $A_u(\Gamma)\sim A_v(\Gamma)$ if and only if $u^{\rm cycle}_C = v^{\rm cycle}_C$ and $A_u(\Gamma) \cong A_v(\Gamma) \iff u^{\rm cycle}_{C}=v^{\rm cycle}_{C}$ or $u^{\rm cycle}_{C^*}=v^{\rm cycle}_{C}$.
 
 Let $\Bbbk =\C$ and $xRy \iff x=y \text{ or } x^{-1}=y$ ($x,y \in \C^*$). Consider the maps
 \begin{gather*}
 \phi \colon \ \Sigma/_\sim \to \C^*, \qquad [A_v(\Gamma)]_\sim \mapsto v^{\rm cycle}_C, \qquad \psi \colon \ \Sigma/_{\cong} \to \C^*/R, \qquad [A_v(\Gamma)]_{\cong} \mapsto v^{\rm cycle}_C.
 \end{gather*}
 These maps are well-def\/ined and injective by the above remark. Surjectivity follows from the fact that if $x \in \C^*$, then the following is a collection of skew-zigzag coef\/f\/icients:
 \begin{gather*}
 v_{b,b}^a=v_{c,c}^a=v_{a,a}^b=v_{c,c}^b= v_{a,a}^c=v_{b,b}^c=v_{b,c}^a=v_{c,b}^a=v_{a,c}^b=v_{c,a}^b=1, \\ v_{a,b}^c=x, \qquad v_{b,a}^c=x^{-1}.
 \end{gather*}
 Thus, $\Sigma/_\sim \cong \C^*$, whereas $\Sigma/_{\cong} \cong \C^*/R$. Note that, $\C^* \ncong \C^*/R$ since, in $\C^*/R$, every element is its own inverse which is not true in $\C^*$. 
 
 \details{It is straightforward to calculate that the following is a collection of skew-zigzag coef\/f\/icients:
 \begin{gather*}
 v_{b,b}^a=v_{c,c}^a=v_{a,a}^b=v_{c,c}^b= v_{a,a}^c=v_{b,b}^c=v_{b,c}^a=v_{c,b}^a=v_{a,c}^b=v_{c,a}^b=1, \\
  v_{a,b}^c=2, \qquad v_{b,a}^c=1/2.
 \end{gather*}
 Since $v^{\rm cycle}_C= (v^{-1})^{\rm cycle}_{C^*}$, we have $A_v(\Gamma)\cong A_{v^{-1}}(\Gamma)$. However, $v^{\rm cycle}_C \neq (v^{-1})^{\rm cycle}_{C}$ and so $A_v(\Gamma)\not \sim A_{v^{-1}}(\Gamma)$.}
\end{eg}

\section{Other constructions of some skew-zigzag algebras} \label{sec:literature}

We conclude with a discussion of other constructions of certain skew-zigzag algebras that have appeared in the literature~\cite{MR2988902,COM:9168417}.

Let $\Gamma = (V,E)$ be a connected graph. We call $\Omega= (\epsilon_{a,b} \in \Bbbk^* \,|\, a,b \in V)$ a collection of \emph{orientation coefficients} if, for any pair of vertices $a,b \in V$, we have $\epsilon_{a,b}=0$ if $\{a,b\} \notin E$ and $\epsilon_{a,b}=-\epsilon_{b,a}$ if $\{a,b\} \in E$. If $\Gamma$ has at least 3 vertices, we def\/ine the algebra $B^{\Gamma}_{\Omega}$ to be the quotient algebra of the path algebra of $D\Gamma$ by the two sided ideal $I_{\Omega}$ generated by elements of the form
\begin{itemize}\itemsep=0pt
 \item $(a \,|\, b \,|\, c)$ for $\{a,b\},\{b,c\} \in E$ and $a \neq c$, and
 \item $\epsilon_{a,b}(a \,|\, b \,|\, a) - \epsilon_{a,c}(a \,|\, c \,|\, a)$ such that $a$ is connected to both $b$ and $c$.
\end{itemize}

\begin{lem} \label{ep to u}
 Let $\Gamma$ be a connected graph. For any collection of orientation coefficients, $\Omega$, there exists a collection of skew-zigzag coefficients, $v$, such that $B^{\Gamma}_{\Omega} = A_v(\Gamma)$.
\end{lem}

\begin{proof}
 Let $\Omega = (\epsilon_{a,b})$ be a collection of orientation coef\/f\/icients and set $v_{b,c}^{a} = \epsilon_{a,c}/\epsilon_{a,b}$ for all $\{a,b\},\{a,c\} \in E$. For any $\{a,b\},\{a,c\},\{a,d\} \in E$ we have
 \begin{gather*}
 v_{b,b}^{a}=\frac{\epsilon_{a,b}}{\epsilon_{a,b}}=1,\qquad
 v_{b,c}^{a}v_{c,b}^{a}=\frac{\epsilon_{a,c}}{\epsilon_{a,b}}\frac{\epsilon_{a,b}}{\epsilon_{a,c}}=1, \qquad
 \text{and} \qquad
 v_{b,c}^{a}v_{c,d}^{a}v_{d, b}^{a}=\frac{\epsilon_{a,c}}{\epsilon_{a,b}}\frac{\epsilon_{a,d}}{\epsilon_{a,c}}\frac{\epsilon_{a,b}}{\epsilon_{a,d}}
 =1.
 \end{gather*}
 Thus the set $v=(v_{b,c}^{a} \,|\, \{a,b\},\{a,c\} \in E )$ is a collection of skew-zigzag coef\/f\/icients. It is clear that $I_{\Omega}=I_v$. Thus, we have $B^{\Gamma}_{\epsilon}=A_v(\Gamma)$.
\end{proof}

\begin{defin}[orientation]
 Let $\Gamma = (V,E)$ be a connected graph and let $D\Gamma=(V,E')$ be its double graph. A set $\epsilon \subseteq E'$ is said to be an \emph{orientation} of $D\Gamma$ if, for every $\{a,b\} \in E$, exactly one directed edge in $D\Gamma$ between $a$ and $b$ is in $\epsilon$.
\end{defin}

In \cite[Section~2.1, p.~109]{COM:9168417} the authors f\/ix an orientation $\epsilon$ of $D\Gamma$. Then, they def\/ine orientation coef\/f\/icients $\Omega = (\epsilon_{a,b} \,|\, a,b \in V)$ as follows:
\begin{gather} \label{ep orientation}
 \epsilon_{a,b} =
 \begin{cases}
 1 & \text{if } (a \,|\, b) \in \epsilon, \\
 -1 & \text{if } (b \,|\, a) \in \epsilon, \\
 0 & \text{if $a$ and $a$ are not connected.}
 \end{cases}
\end{gather}
Notice that $\epsilon_{a,b}=-\epsilon_{b,a}$. Thus, by Lemma \ref{ep to u}, we have $B_{\Omega}^{\Gamma} = A_v(\Gamma)$, where
\begin{gather*}
 v=\left(v^a_{b,c}= \frac{\epsilon_{a,c}}{\epsilon_{a,b}} \,|\, \{a,b\},\{a,c\} \in E\right).
\end{gather*}
However, the following proposition shows that the converse of Lemma~\ref{ep to u} is false. Thus, the alternate def\/inition of skew-zigzag algebras in terms of orientation coef\/f\/icients is more restrictive.

\begin{prop} \label{not iso}
 If $\Gamma$ is not a bipartite graph, then $A(\Gamma)$ is not isomorphic to $B_\Omega^\Gamma$ for any collection of orientation coefficients~$\Omega$.
\end{prop}

\begin{proof}
 Suppose that $\Gamma$ is not a bipartite graph and that there exists an isomorphism
 \begin{gather*}
 \phi \colon \ A(\Gamma) \to B^{\Gamma}_{\Omega}
 \end{gather*}
 for some collection of orientation coef\/f\/icients $\Omega$. For all $a \in V$ let $\phi([a])=[x_a]$. So, for all \mbox{$\{a,b\} \in E$} we have $\phi([a \,|\, b])=\alpha_{a,b}[x_a \,|\, x_b]$ for some $\alpha_{a,b} \in \Bbbk$. Consequently, for any $\{a,b\}$, $\{a,c\} \in E$ we have
 \begin{gather*}
 \alpha_{a,b}\alpha_{b,a}[x_a \,|\, x_b \,|\, x_a]=\phi([a \,|\, b \,|\, a])
 =\phi([a \,|\, c \,|\, a]) = \alpha_{a,c}\alpha_{c,a}[x_a \,|\, x_c \,|\, x_a].
 \end{gather*}
 Therefore,
 \begin{gather} \label{epsilons1}
 \frac{\alpha_{a,c}\alpha_{c,a}}{\alpha_{a,b}\alpha_{b,a}}=\frac{\epsilon_{x_a,x_c}}{\epsilon_{x_a,x_b}}.
 \end{gather}
 By \cite[Proposition~1.6.1]{MR2744811}, since $\Gamma$ is not bipartite, it contains a~cycle of odd length. Thus, there is a~cycle $C=(a_1 \,|\, \ldots \,|\, a_n)$ in $D\Gamma$ with $n$ even. So, \eqref{epsilons1} yields
 \begin{gather}
 1=\frac{\alpha_{a_1,a_2}\alpha_{a_2,a_1}}{\alpha_{a_1,a_{n-1}}\alpha_{a_{n-1},a_1}} \frac{\alpha_{a_2,a_3}\alpha_{a_3,a_2}}{\alpha_{a_2,a_{1}}\alpha_{a_1,a_{2}}} \cdots \frac{\alpha_{a_{n-1},a_1}\alpha_{a_1,a_{n-1}}}{\alpha_{a_{n-1},a_{n-2}}\alpha_{a_{n-2},a_{n-1}}}\nonumber\\
 \hphantom{1}{}
 = \frac{\epsilon_{x_{a_1},x_{a_2}}}{\epsilon_{x_{a_1},x_{a_{n-1}}}} \frac{\epsilon_{x_{a_2},x_{a_3}}}{\epsilon_{x_{a_2},x_{a_{1}}}} \cdots \frac{\epsilon_{x_{a_{n-1}},x_{a_1}}}{\epsilon_{x_{a_{n-1}},x_{a_{n-2}}}}.\label{cycle epsilon}
 \end{gather}
 Since $\{a_i,a_{i+1}\} \in E$ for all $1\leq i \leq n-1$, $\epsilon_{x_{a_j},x_{a_k}}\neq 0$ for $1\leq j, k, \leq n-1$, $j \neq k$. Moreover, we know that $\epsilon_{a,b}=-\epsilon_{b,a}$ for all $a,b \in V$. Consequently, \eqref{cycle epsilon} yields
 \begin{gather*}
 1=(-1)^{n-1} = -1.
 \end{gather*}
 This contradiction implies that $A_u(\Gamma) \ncong B^{\Gamma}_{\Omega}$ for any $\Omega$.
\end{proof}

Suppose $(\epsilon_{a,b})$ is a collection of orientation coef\/f\/icients. For $x \in V$, def\/ine $V_{x}$ and $y_x$ as in Proposition \ref{basis}. Notice that if we modify the set $J$ in \eqref{J} by setting
\begin{gather*}
 J' \coloneqq \{[a], [b \,|\, c], \epsilon_{x,y_x}[x \,|\, y_x \,|\, x] \,|\, a,x \in V,\text{ } b,c \in V \text{ such that } \{b,c\} \in E \},
\end{gather*}
then $J'$ is independent of the choice of $y_x$ for every $x \in V$. Indeed, for any $y,z \in V_{x}$, we have
\begin{gather*}
 \epsilon_{x,y}[x \,|\, y \,|\, x]=\epsilon_{x,z}[x \,|\, z \,|\, x].
\end{gather*}
In \cite[Section~6.1, p.~2516]{MR2988902}, the authors f\/ix an orientation $\epsilon$ of $\Gamma$ and def\/ine coef\/f\/icients $\epsilon_{a,b}$ $(a,b \in V)$ as in \eqref{ep orientation}. They then def\/ine a diagrammatic algebra using these coef\/f\/icients. Their algebra is in fact isomorphic to the algebra $B_{\Omega}^{\Gamma}$, for $\Omega = (\epsilon_{a,b} \,|\, a,b \in V)$, via the map
\begin{gather*}
 A \to A(\Gamma), \qquad
 \begin{tikzpicture}[baseline={([yshift=-.5ex]current bounding box.center)},>=stealth',auto,node distance=3cm, thick,main node/.style={circle,draw,font=\sffamily\Large\bfseries}]
 \draw[thick,->] (0,0) -- (0,2);
 \filldraw (0,1) circle (2pt);
 \draw (0,2) node [anchor=south] {$y$};
 \draw (0,0) node [anchor=north] {$x$};
 \end{tikzpicture}
 \mapsto
 \begin{cases}
 [x \,|\, y] & \text{$x\neq y$,} \\
 \epsilon_{x,y_x}[x \,|\, y_x \,|\, x] & \text{if $x=y$},
 \end{cases}
 \qquad
 \begin{tikzpicture}[baseline={([yshift=-.5ex]current bounding box.center)},>=stealth',auto,node distance=3cm, thick,main node/.style={circle,draw,font=\sffamily\Large\bfseries}]
 \draw[thick,->] (0,0) -- (0,2);
 \draw (0,1) node [anchor=west] {$x$};
 \end{tikzpicture}
 \mapsto [x].
\end{gather*}
\details{Let $\phi$ be the given map, which is clearly bijective. Furthermore, we have
 \begin{gather*}
 [x \,|\, y][y \,|\, z]=[x \,|\, y \,|\, z]=0,
 \end{gather*}
 if $x \neq z$. If $x=z$, we have
 \begin{gather*}
 [x \,|\, y][y \,|\, z]=[x \,|\, y \,|\, x]=v_{y,y_x}^x[x \,|\, y_x \,|\, x]=\frac{\epsilon_{x,y_x}}{\epsilon_{x,y}}[x \,|\, y_x \,|\, x].
 \end{gather*}
 Since $\epsilon_{x,y}= \pm 1$ for all $x,y \in V$, we have $\epsilon_{x,y}=\epsilon_{x,y}^{-1}$. Thus,
 \begin{gather*}
 [x \,|\, y][y \,|\, z]=\epsilon_{x,y}\epsilon_{x,y_x}[x \,|\, y_x \,|\, x].
 \end{gather*}
 Therefore, we have
 \begin{gather*}
 [x \,|\, y][y \,|\, z]= \delta_{xz} \epsilon_{xy} \epsilon_{x,y_x}[x \,|\, y_x \,|\, x].
 \end{gather*}
Consequently, $\phi$ is an algebra homomorphism.}
So, by Lemma~\ref{ep to u}, $A$ is isomorphic to $A_v(\Gamma)$ where $v=\big(v^a_{b,c}= \frac{\epsilon_{a,c}}{\epsilon_{a,b}} \,|\, \{a,b\},\{a,c\} \in E\big)$. However, if $\Gamma$ is not bipartite, then this is not isomorphic to the zigzag algebra by Proposition~\ref{not iso}.

\textbf{Note on the \LaTeX version.} For the interested reader, the tex f\/ile 
of this paper includes hidden details of some straightforward computations and arguments that are omitted in the pdf f\/ile. These details can be displayed by switching the \texttt{details} toggle to true in the tex f\/ile and recompiling.

\subsection*{Acknowledgements}
 This work was completed under the supervision of Professor Alistair Savage. The author would like to thank Professor Savage immensely for his patience and guidance throughout this paper as well as the opportunity to write this paper. The author would also like to thank the University of Ottawa and the Work-Study Program for their support. Finally, the author would like to thank the referees for their useful comments and for providing a~reference for Proposition~\ref{symmetric algebra}.

\pdfbookmark[1]{References}{ref}
\LastPageEnding

\end{document}